\documentclass[11pt]{amsart}
\usepackage[left=3cm,right=3cm,top=3cm,bottom=3cm,a4paper]{geometry}   	 
\usepackage{amscd, amssymb, mathrsfs, mathabx, tikz-cd, amsthm, thmtools, stmaryrd, bm}
       
\usepackage{hyperref}

\input{xy}
\xyoption{all}

\numberwithin{equation}{section}

\newcommand{\CC}{\mathbb{C}}

\newcommand{\EE}{\mathbb{E}}

\newcommand{\PP}{\mathbb{P}}
\newcommand{\QQ}{\mathbb{Q}}

\newcommand{\TT}{\mathbb{T}}

\newcommand{\ZZ}{\mathbb{Z}}

\newcommand{\cal}{\mathcal}

\newcommand\cC{{\cal C}}
\newcommand\cD{{\cal D}}

\newcommand\cL{{\cal L}}

\newcommand\cN{{\cal N}}
\newcommand\cO{{\cal O}}
\newcommand\cP{{\cal P}}

\newcommand\cV{{\cal V}}
\newcommand\cW{{\cal W}}

\def\sta{{^{\ast}}}

\def\and{\quad{\rm and}\quad}

  \DeclareMathOperator{\Aut}{Aut}
 
 \DeclareMathOperator{\val}{val}

\DeclareMathOperator{\vdim}{vdim}

\def\vphi{\varphi}

\def\hra{\hookrightarrow}

\def\lra{\longrightarrow }
\def\bar{\overline}

\def\wtil{\widetilde}
\def\what{\widehat}

\def\deg{\mathrm{deg}\,}
\def\dim{\mathrm{dim}}

\def\rH{\mathrm{H}}

\def\c>{\succ}
\def\c<{\prec}

\def\l({\left(}
\def\r){\right)}

 \def\ZZ{\mathbb{Z}}
\newtheorem{prop}{Proposition}[section]

\newtheorem{lemm}[prop]{Lemma}

\newtheorem{rema}[prop]{Remark}

\newtheorem{defi}[prop]{Definition}

\def\dual{^{\vee}}
\def\Ob{{\mathcal Ob}}
\def\virt{^{\mathrm{vir}} }
\def\loc{_{\mathrm{loc}} }

\def\bL{\bold{L}}
\def\fl{\mathrm{fl}}
\def\ev{\mathrm{ev}}
\def\bx{\bold{x}}
\def\by{\bold{y}}

\newtheorem*{conv}{Convention}
\newtheorem*{caut}{Caution}

\newcommand{\set}[1]{\{ #1 \}}

\def\bL{{\bf{L}}}

\def\log{\mathrm{log}}

\title{Irregular vanishing on $\PP^2 \times \PP^2$}

\author[Chang]{Huai-Liang Chang}
\address{Institute for Math and AI, Wuhan University, Wuhan 430072, China, and Department of Mathematics, HKUST, Hong Kong}
\email{hlchang@whu.edu.cn}

\author[Lee]{Sanghyeon Lee}
\address{Department of Mathematics, Ajou University, Suwon, South Korea
}
\email{sanghyeon25@ajou.ac.kr}

\author[Li]{Jun Li}
\address{Shanghai Center for Mathematical Sciences(SCMS), Fudan University, Shanghai, China, and Shanghai Institute for Mathematics
and Interdisciplinary Sciences(SIMIS), Shanghai, China}
\email{lijun2210@fudan.edu.cn}

\thanks{}

\date{}

\begin{document}

\begin{abstract}
In this paper, we describe Mixed-Spin-P(MSP) fields for a smooth CY 3-fold $X_{3,3} \subset \mathbb{P}^2 \times \mathbb{P}^2$. Then we describe $\mathbb{C}^* -$fixed loci of the moduli space of these MSP fields. We prove that any virtual localization term coming from the fixed locus corresponding to an irregular graph does not contribute to the invariant if the graph is not a pure loop, and also prove this vanishing property for the moduli space of N-MSP fields.
\end{abstract}

\maketitle

\setcounter{tocdepth}{1}
\tableofcontents

\section{Introduction}

  In this paper, we introduce the theory of Mixed-Spin-P(MSP) fields for the smooth Calabi-Yau 3-fold $X_{3,3} \subset \PP^2 \times \PP^2$, which is the zero set of a bidegree $(3,3)$ polynomial. This is a generalization of MSP fields for the quintic 3-fold $Q\subset \PP^4$ introduced in \cite{CGLL}.
  In the quintic 3-fold case, the moduli space of MSP fields $\cW_{g,\bold{d},\gamma}$ contains the moduli space of $p$--fields $M_{g,k}^p(\PP^4,d)$ developed in \cite{CL12} and the moduli space of $\mu_5$--spin curves $M_{g,d,\gamma}(\mu_5)$ developed in \cite{FJR1, FJR2, CLL15} as its sublocus. 

\smallskip

On the other hand, there is a torus action on the moduli space $\cW_{g,\bold{d},\gamma}$ and its torus fixed locus $\cW_{g,\bold{d},\gamma}^T$ is a union of $M_{g,k}^p(\PP^4,d)$ and $M_{g,d,\gamma}(\mu_5)$, and other fixed loci. In fact, the moduli space $\cW_{g,\bold{d},\gamma}$ has a $T$--equivariant perfect obstruction theory and $T$--invariant cosection. Furthermore, its cosection degeneracy locus is proper \cite{CLLL2,CLLL}.

\smallskip

Furthermore, a notion, which is a modification of the notion of MSP fields, called \textbf{`N-MSP'} fields is developed in \cite{CGLL}. The moduli space of N-MSP fields is similar to the moduli space of MSP fields, and also has good properties. This new notion turns out to be useful to compute enumerative invariants, such as Gromov-Witten(GW) invariants and FJRW invariants.
  In \cite{CGL1,CGL2}, the authors applied torus localization to the localized virtual cycle of the moduli space $[\cW_{g,\bold{d},\gamma}]\virt\loc$ to prove polynomiality and BCOV's Feynman rule of the Gromov-Witten potential function of the quintic 3-folds. 
  When they package terms from torus fixed loci, it is important to show that the contributions from fixed loci corresponding to \textbf{`irregular'} graphs vanishes to make the computation simple. It is proven in \cite{CL20van, CGLL}. We also note that above works for quintic 3-folds have recently been generalized to the case of Calabi-Yau hypersurfaces in weighted $\PP^4$s in \cite{Lei1, Lei2, Lei3}.

\smallskip

In this paper, we will describe the torus fixed loci of the moduli space of MSP fields for $X_{3,3}$ in Section \ref{sect:MSP}. Then we will show that contributions from the fixed loci corresponding to irregular graphs with $0$--$\infty$ edge vanish in the torus localization computation in Section \ref{sect:insepvan}. 
Proving this vanishing will allow us to reduce the localization computation to the fixed loci corresponding to graphs where there are no ($0$--$\infty$)--edges, as in \cite{CGL1,CGL2, Lei2, Lei3}. For the moduli space of N--MSP fields, we also prove the same vanishing in Section \ref{sect:insepvanN}. 

\smallskip

We expect that this will help to prove polynomiality and BCOV's Feynman rule for the smooth CY 3-fold $X_{3,3} \subset \PP^2 \times \PP^2$. 
On the other hand, different from the quintic 3-fold, the Picard group of $X_{3,3}$ has two generators. I expect the arguments in this paper will help for the similar discussions for CY3 complete intersections in the product of projective spaces.
  We note that descriptions of MSP fields are dealt with in \cite{CGLLZ} in a general setting, for Calabi-Yau 3-folds embedded in a GIT quotient as a complete intersection. In that paper, the authors also proved that moduli spaces are equipped with a natural torus action, $T$--equivariant perfect obstruction theory and $T$--invariant cosection, and with a well-defined cosection localized virtual cycle. Here the $T$--invariant cosection is defined in a parallel way as in \cite{CLLL2}.

\medskip

\noindent \textbf{Acknowledgement.} 
The second named author thanks Yang Zhou for intensive discussions on $\CC^*$--action and their fixed loci on the moduli space of MSP fields. The first named author was partially supported by NSFC grants 24SC04, and Hong Kong grant GRF 16303122 and 16301524. 
The second named author thanks Fudan University and Shanghai Center for Mathematical Sciences for the excellent working environment provided during the writing of the paper. The third named author is partially supported by the National Key Research and Development Program of China $\#$2020YFA0713200, NSFC grant 12071079, and by Shanghai SF grant 22YS1400100. The third named author would like to acknowledge the support as a member from the Key Laboratory of Mathematics for Nonlinear Sciences.

\section{Description of MSP fields}\label{sect:MSP} 
In this section, we will describe theory of Mixed-Spin-P(MSP) fields for the Calabi-Yau 3-fold $X_{3,3}$. 

\subsection{The Landau-Ginzburg model} $\phantom{a}$ \\
\indent Consider the following Landau-Ginzburg(LG) model. Let 
$V = \CC^3 \times \CC^3 \times \CC \times \CC \times \CC$, and denote its coordinates by $(x_1,x_2,x_3,y_1,y_2,y_3,p,u,v)$ and let $\bx := (x_1,x_2,x_3)$ and $\by:=(y_1,y_2,y_3)$. Let $G := \CC^* \times \CC^* \times \CC^*$, $R:= \CC^*$, $\Gamma := G \times R$. 
We let $\Gamma$ act on $V$ by weights:
\[
\left[
\begin{matrix}
1 & 1 & 1 & 0 & 0 & 0 & -3 & 1 & 0 \\
0 & 0 & 0 & 1 & 1 & 1 & -3 & 0 & 0 \\
0 & 0 & 0 & 0 & 0 & 0 & 0 & 1 & 1 \\
\hline
0 & 0 & 0 & 0 & 0 & 0 & 1 & 0 & 0
\end{matrix}
\right].
\]

Note that the first three rows are the weight of $G = (\CC^* \times \CC^* \times \CC^*)$--action on $V$, and the last row consists of the weights of $R$--action on $V$. The potential function $V \to \CC$ is given by
\[
F(\bx,\by) \cdot p
\]
where $F$ is the bi-homogeneous function with degree $3$ in $\bx$ and with degree $3$ in $\by$. We assume that $F$ defines a smooth hypersurface $X_{3,3}$ in $\PP^2 \times \PP^2$.

\smallskip

Pick the character $\vartheta : \Gamma \to \CC^*$ given by the weight vector $(1,1,2,0)$ and let $\theta := \vartheta|_G$. Let us consider the GIT quotient $\wtil{X} := V \sslash_\theta G$. Then the unstable locus of this $G$-action on $V$ via the character $\theta$ is given by the common vanishing locus of the functions:
\[
x_i y_j v^2, \ \ y_j uv, \ \ pu^6 y_j^6, \ \ \ \  1 \leq i,j \leq 3.
\]
This locus is equal to:
\begin{align}\label{eq:unstable}
\left\{ (y_1,y_2,y_3)=0 \right\} \cup \left\{ (u,v) = 0 \right\} \cup \left\{ (x_1,x_2,x_3,u) = 0 \right\} \cup \left\{ (p,v)=0 \right\}.
\end{align}


Therefore, we have a projection
\[
\wtil{X}= [V \sslash_\theta G] \to [\CC^3- \left\{ 0\right\} / \CC^*]=\PP^2
\]
where $\CC^3$ is the factor of $V$, corresponding to the coordinates $(y_1,y_2,y_3)$.
Next we consider a $\CC^*$--action on $\wtil{X}$, given by
\[
t \cdot (\bx,\by,p,u,v) := (\bx,\by,p,tu,v) 
\]
Then, we can observe that the fixed locus of this action is given by
\[
[\set{ u=0 }\sslash_\theta G] \sqcup [\set{v=0} \sslash_\theta G] \sqcup [\set{u\neq 0, v\neq 0, \bx = 0, p=0} \sslash_\theta G].
\]
We have 
\begin{align*}
& X_0 := [\set{u=0} \sslash_\theta G] \cong K_{\PP^2 \times \PP^2} \\
& X_1 := [\set{u\neq 0, v\neq 0, \bx = 0, p=0} \sslash_\theta G] \cong \PP^2 \\
& X_\infty := [\set{v=0} \sslash_\theta G] \textrm{ is a fibration over $\PP^2$ whose fibers are isomorphic to $[\CC^3/\mu_3]$}.
\end{align*}

\noindent Moreover, the restriction of the potential function
\[
F(\bx, \by)\cdot p \, |_{X_0} : X_0 \to \CC
\]
gives us a Calabi-Yau Landau-Ginzburg(LG) model, whose critical locus is isomorphic to $X_{3,3}:=\set{F=0} \subset \PP^2 \times \PP^2$.
  On the other hand, the restriction of the potential function
\[
F(\bx, \by)\cdot p \, |_{X_\infty} : X_\infty \to \CC
\]
gives us a hybrid LG model.

\subsection{Moduli space of Mixed-Spin-P fields}

We construct the moduli space of Mixed-Spin-P(MSP) fields, whose target is $\wtil{X}$.
It is similar to the moduli space of stable maps to $\wtil{X}$, but slightly different from that. An MSP field to $\wtil{X}$ is a datum:
$(\cC, \Sigma^{\cC}, \cL_1, \cL_2, \cN, \phi, \theta, \rho, \mu, \nu )$ which satisfies the following:
\begin{itemize}
\item[•]
$\cC$ is a genus $g$ nodal twisted (orbifold) curve.
\item[•]
$\Sigma^{\cC}$ is a set of marked points on $\cC$, $\cC \setminus (\Sigma^{\cC} \sqcup \set{\textrm{nodes}})$ is a scheme.
\item[•]
For $p \in \Sigma^{\cC}$, the automorphism group $\Aut(p)$ is isomorphic to $\set{id}$ or $\mu_3(\cong \ZZ_3)$.
\item[•]
$\cL_1,\cL_2,\cN$ are line bundles on $\cC$. 
\item[•]
For $p \in \Sigma^{\cC}$ $\Aut(p)$ acts faithfully on $\cL_1|_p$. For $\zeta_3 = \exp(2\pi i/3)$ and $v \in \cL_1|_p$, $\zeta_3 \cdot v = \zeta_3^{m(p)} v$, $m(p) \in \set{0,1,2}$.
\item[•]
We give a decomposition $\Sigma^{\cC} = \sqcup_{i=0}^2 \Sigma^{\cC}_i$ such that every $p \in \Sigma^{\cC}_i$ satisfies $m(p)=i$. 
\item[•]
$\phi=(\phi_1,\phi_2,\phi_3) \in \rH^0( \cC, \cL_1 ^{\oplus 3} ), \ \theta=(\theta_1,\theta_2,\theta_3) \in \rH^0(\cC, \cL_2^{\oplus 3}), \ \rho \in \rH^0(\cC, \cL_1^{-3}\otimes \cL_2^{-3} \otimes \omega^{\log}_{\cC} ), \ \mu \in \rH^0(\cC, \cL_1 \otimes \cN), \nu \in \rH^0( \cC, \cN)$.
\item[•]
There is a decomposition $\Sigma^{\cC}_0 = \Sigma^{\cC}_{(1,\phi)} \sqcup \Sigma^{\cC}_{(1,\rho)}$ such that 
$\rho|_{ \Sigma^{\cC}_{(1,\rho)}}=0, \ \phi|_{\Sigma^{\cC}_{(1,\phi)}}=0$. Therefore $ \ \phi_1,\phi_2,\phi_3 \in \rH^0( \cC, \cL_1(-\Sigma^{\cC}_{(1,\phi)}) ), \ \rho \in \rH^0(\cC, \cL_1^{-3}\otimes \cL_2^{-3} \otimes \omega^{\log}_{\cC}(-\Sigma^{\cC}_{(1,\rho)}) ).$
\end{itemize}

Let $\Sigma^{\cC} = \set{x_1,\dots,x_n}$ be marked points and let $m_i := m(x_i)$. Note that $\mu_3 = \set{1, \zeta_3, \zeta_3^2} \subset \CC^*$. Let $\mu_{na} := \set{(1,\phi),(1,\rho)}\cup (\mu_3 - \set{1})$, $\mu_{br} := \set{(1,\phi),(1,\rho)}\cup \mu_3$. Note that `na' stands for \textbf{`narrow sector'} and `br' stands for \textbf{`broad sector'} in FJRW theory. 
  In this paper, we assume \textbf{`narrow condition'} : Let $\gamma_i := \zeta_3^{m_i}$. If $\gamma_i = 1$($m_i = 0$) then $x_i \in \Sigma^{\cC}_{(1,\phi)} \cup \Sigma^{\cC}_{(1,\rho)}$, so that $\phi(x_i) = 0$ or $\rho(x_i)=0$. If $x_i \in \Sigma^{\cC}_{(1,\phi)}$ then we assign $\gamma_i = (1, \phi)$ instead of $\gamma_i = 1$, and if $x_i \in \Sigma^{\cC}_{(1,\rho)}$ then we assign $\gamma_i = (1, \rho)$ instead of $\gamma_i = 1$. Hence we obtain a vector $\gamma(\xi) = (\gamma_1,\dots, \gamma_n)$, where $\gamma_i \in \mu_{na}$.
  Also we define the stability condition for MSP field $\xi = ( \cC, \Sigma^{\cC}, \cL_1, \cL_2, \cN, \phi, \theta, \rho, \mu, \nu ) $ as follows. This coincides with the stability condition in \cite{CGLLZ} for $X_{3,3}$ case. 

  Let $d_0(\xi):= \deg \cL_1 \otimes \cN$, $d_\infty(\xi) := \deg \cN$, and $d(\xi) := \deg \cL_1 = d_0(\xi) - d_\infty(\xi)$.

\begin{defi}\label{cond:stab}
The MSP field $\xi = ( \cC, \Sigma^{\cC}, \cL_1, \cL_2, \cN, \phi, \theta, \rho, \mu, \nu ) $ is called stable if it satisfies the following conditions
\begin{enumerate}
\item The pairs
\[
(\mu,\nu), \ (\phi,\mu), \ \theta, \ (\rho,\nu)
\]
are nowhere vanishing.
\item 
  The $\QQ$-line bundle  
  \begin{align}\label{cond}
    \cL_1 \otimes \cL_2 \otimes \cN^{\otimes 2} \otimes (\omega_{\cC}^{\log})^{\otimes \left(\frac{1}{3} \right)^+ }
  \end{align}
  is positive, i.e. the degrees are positive over all irreducible subcurves over $\cC$.
\end{enumerate}
\end{defi}

\begin{rema}
Because of the stability condition (1) in Definition \ref{cond:stab}, the triple $\theta=(\theta_1,\theta_2,\theta_3)$ does not vanish everywhere, hence defines a morphism $\bar{\theta} : \cC \to \PP^2$.
\end{rema}



\subsection{$\CC^*$--action on MSP fields, equivariant perfect obstruction theory, and $\CC^*$--fixed loci} \label{sect:equivariant1}
$\phantom{a}$ \\
\indent Next we define a $\CC^*$--action on MSP fields. We define the action by
\[
t \in \CC^*, \ t \cdot (\cC, \Sigma^{\cC}, \cL_1, \cL_2, \cN, \phi, \theta, \rho, \mu, \nu ) := (\cC, \Sigma^{\cC}, \cL_1, \cL_2, \cN, \phi, \theta, \rho, t \mu, \nu ).
\]
  
  We define $\cW_{g,\gamma,\bold{d}}$ as the moduli space of the above MSP fields $(\cC, \Sigma^{\cC}, \cL_1, \cL_2, \cN, \phi, \theta, \rho, \mu, \nu )$. 
  Parallel to the case of the moduli space of MSP fields for the quintic 3-fold established in \cite{CGLLZ}, $\cW_{g,\gamma,\bold{d}}$ is a DM stack.
  We will usually abbreviate $\cW_{g,\gamma,\bold{d}}$ as $\cW$.
  Let $\cD = \cD_{g,\gamma,\bold{d}}$ be the stack of data $(\cC, \Sigma^{\cC}, \cL_1, \cL_2, \cN)$. Parallel to \cite{CL12,CLLL2,CLLL}, we can check that $\cD_{g,\gamma,\bold{d}}$ becomes a smooth Artin stack and also can check that $\cW_{g,\gamma,\bold{d}}$ has a following $\CC^*$--equivariant relative obstruction theory over $\cD_{g,\gamma,\bold{d}}$:
\begin{gather*}
\phi\dual_{\cW/\cD} : \TT_{\cW/\cD} \to \EE_{\cW/\cD} := R \pi_* \cV \\
\textrm{where } \cV := (\cL_1^{\log})^{\oplus 3} \oplus \cL_2^{\oplus 3} \oplus \cP^{\log} \oplus (\cL_1 \otimes \cN \otimes \bL_1) \oplus \cN \\
\textrm{where } \cL_1^\log := \cL_1(-\Sigma^{\cC_\cW}_{(1,\phi)}), \ \cP^{\log} := (\cL_1\dual)^{\otimes 3}\otimes \omega^{\log}_{\cC_\cW / \cW}(-\Sigma^{\cC_\cW}_{(1,\rho)}), \\
\cC_{\cW} \to \cW \textrm{ is the universal curve}, \textrm{ $\bL_k$ is the representation of $\CC^*$ on $\CC$ with a weight $k$}.
\end{gather*}

Moreover, parallel to \cite{CL12} and \cite[Section 2.3, equation (2.9)]{CLLL2}, we can construct a ($\CC^*$--equivariant) cosection $\Ob_\cW := \rH^1(\EE_{\cW/\cD}) \to \cO_\cW$ and check that the degeneracy locus of the cosection, $\cW^- \subset \cW$, is proper.


\medskip

Let $\cW^T \subset \cW$ be the fixed locus of the above $\CC^*$--action and we will describe its property in this section. Consider an MSP field
\[
  \xi = (\cC, \Sigma^{\cC}, \cL_1, \cL_2, \cN, \phi, \theta, \rho, \mu, \nu).
\]
If we take a suitable finite cover of the torus $\CC^* \to \CC^*$, $t \mapsto t^\ell$, then if $\xi$ is $\CC^*$--fixed (which means $t^\ell \cdot \xi \cong \xi$ for all $t \in \CC^*$), one can always find a morphism $h : \CC^* \to \Aut(\cC, \Sigma^{\cC})$ and linearizations $\tau_{t,\cL_1} : h(t)_* \cL_1 \to \cL_1$, $\tau_{t,\cL_2} : h(t)_* \cL_2 \to \cL_2$, and $\tau_{t,\cN} : h(t)_* \cN \to \cN$ such that
\begin{align*}
t \cdot (\phi, \theta, \rho, \mu, \nu) & = (\phi, \theta, \rho, t\mu, \nu) \\
& =(\tau_{z,\cL_1},\tau_{z,\cL_2},\tau_{z,\cN})(h(z)_*\phi, h(z)_* \theta, h(z)_*, h(z)_* \rho, h(z)_* \mu, h(z)_* \nu).
\end{align*}

Note that by an argument parallel to \cite[Section 2.2]{CLLL}, for a stable $\CC^*$--fixed MSP field $\xi$, we can show that the induced $\CC^*$--action on the domain curve $\cC$ and the linearizations $\tau_{t, \cL_1}, \tau_{t, \cL_2}, \tau_{t, \cN}$ are unique. For convenience, we allow our $T = \CC^*$ to act on curves with rational weights, instead of considering the finite cover of $T$.

\subsubsection{Decomposition of the domain curve} \label{sect:decompdomain}$\phantom{a}$ \\
\indent For $\xi \in \cW^T$, we will decompose $\cC = \cC_\xi$ by components which are stable under the $\CC^*$--action. Similar to \cite{CLLL}, we have the decomposition
\[
\cC = \cC_0 \cup \cC_1 \cup \cC_\infty \cup \cC_{01} \cup \cC_{1\infty } \cup \cC_{0\infty}.
\]
Each component $\cC_i$, $\cC_{ij}$ is defined by the following:
\begin{align*}
& \cC_0 := \set{ \mu = 0, \nu \neq 0 }, \ \cC_1 := \set{\mu \neq 0, \nu \neq 0, \phi = \rho = 0}, \ \cC_\infty := \set{\mu \neq 0, \nu = 0},  \\
& \cC_{01} = \set{\rho=0, \nu \neq 1}, \ \cC_{1\infty} = \set{\mu \neq 0, \phi = 0}. 
\end{align*}


\subsubsection{Localization graph corresponding to a $\CC^*$--fixed MSP field} \label{sect:assocgraph} $\phantom{a}$ \\
\indent Using the decomposition in Section \ref{sect:decompdomain}, for $\xi \in \cW^T$, we assign a localization graph $\Gamma_\xi$ with the following data of vertices, edges, legs:
\begin{itemize}
\item[(1)]
Set of vertices $V(\Gamma_\xi) := V_0 \cup V_1 \cup V_\infty$, where $V_i$ is the set of connected components of $\cC_i$, $i=0,1,\infty$. 
Therefore, for each vertex $v$, there is the corresponding connected curve $\cC_v \subset \cC$.
\item[(2)]
Set of edges $E(\Gamma_\xi) := E_{01} \cup E_{1\infty} \cup E_{0 \infty}$, where $E_{ij}$ are irreducible components($\cong \PP^1$ or $\PP(1,3)$) of $\cC_{ij}$, $i,j \in \set{0,1,\infty}, \ i\neq j$. For each edge $e$, there is a corresponding irreducible curve $\cC_e \subset \cC$ and $\cC_e \cong \PP^1$ or $\PP(1,3)$. Moreover, $e \in E_{ij}$ connects a vertex $v_i \in V_i$ and a vertex $v_j \in V_j$ where $\cC_e$ intersects with $\cC_{v_i}, \cC_{v_j}$.
\item[(3)]
(Ordered) set of legs $S(\Gamma_\xi)$ is identified with the set of markings $\Sigma^{\cC}$ of $\cC_i$. For $s \in S(\Gamma_\xi) \cong \Sigma^{\cC}$, there is a unique $v \in V(\Gamma_\xi)$ such that $s \in \cC_v$. Then the leg $s$ is connected to a vertex $v$. For each $s \in S$ we assign the marking $x(s) \in \Sigma^\cC$.
Decomposition of $\Sigma^{\cC}$ automatically gives the decomposition $S = S^0 \sqcup S^1 \sqcup S^2$, $S^0 = S^{(1,\phi)} \sqcup S^{(1,\rho)}$.
\end{itemize}
Also the graph is decorated by the following data:
\begin{itemize}
\item[(i)](genus) For each vertex $v \in V$, we assign $g(v) := g(\cC_v) \in \ZZ_{\geq 0}$.
\item[(ii)](degree) For each vertex $v \in V(\Gamma_\xi)$(resp. edge $e \in E(\Gamma_\xi)$) we assign a degree pair $(d_{v 0},d_{v \infty})$ (resp. $(d_{e 0},d_{e \infty} )$) where $d_{v 0}:=\deg \cL_1 \otimes \cN |_{\cC_v}$, $d_{v \infty}:=\deg \cN|_{\cC_v}$(resp. $d_{e 0}:=\deg \cL_1 \otimes \cN |_{\cC_e}$, $d_{e \infty} :=\deg \cN|_{\cC_e}$). We also define $d_v := d_{v0} - d_{v \infty}$ and $d_e := d_{e 0} - d_{e \infty}$.
\item[(iii)](monodromy)
For each leg $s \in S$, we assign $m(s) \in \set{0,1,2}$ such that $\exp(2 \pi i/3)=:\zeta_3 \in \mu_3\cong \Aut(x(s))$ act on $\cL_1|_{x(s)}$, as a scalar multiplication of $\zeta_3^{m(s)}$. 
\end{itemize}
The data above define the localization graph $\Gamma_\xi$ corresponding to $\xi \in \cW^T$. For later use, we will introduce more definitions here.
\begin{itemize}
\item[•]
For a vertex $v \in V$, $E_v \subset E$ is a set of edges connected to $v$. Let $S_v \subset S$ be a set of legs connected to $v$. Then the valency of the vertex $\val(v) = |E_v| + |S_v|$.
\item[•] We call a vertex $v \in V(\Gamma_\xi)$ stable if $\cC_v$ is 1-dimensional. And otherwise we call $v$ unstable. Let $V^S$ be the set of stable vertices and $V^U$ be the set of unstable vertices, and we have $V(\Gamma_\xi) = V^S \cup V^U$.
\item[•]
For $m=0,1,2$, we let $S^m$ be the set of vertices with monodromy $m$. We have the decomposition $S^0 = S^{(1,\phi)} \sqcup S^{(1,\rho)}$ such that $S^{(1,\phi)} \cong \Sigma^\cC_{(1,\phi)}$ and $S^{(1,\rho)} \cong \Sigma^\cC_{(1,\rho)}$. We also define $S_v^m$, $S_v^{(1,\phi)}$, $S_v^{(1,\rho)}$ as well.
\end{itemize}

\noindent Moreover, for a decorated graph $\Gamma$, we assign the following data as follows:
\begin{itemize}
\item[(i)] (genus) We define the genus of the graph $\Gamma$ by $g(\Gamma) := \sum_{v \in V(\Gamma)} g(v) + g(|\Gamma|)$, where $|\Gamma|$ is a graph obtained from $\Gamma$ by forgetting all decorations.
\item[(ii)] (degree) We define the degree of the graph $\Gamma$ by $\bold{d}(\Gamma) = \sum_{v \in V(\Gamma)} \bold{d}(v) + \sum_{e \in E(\Gamma)} \bold{d}(e)$.
\item[(iii)] (monodromy) We define the monodromy vector of the graph $\Gamma$ as follows. Let the set of legs $S(\Gamma) := (s_1,\dots,s_\ell)$, we define $\gamma_i$, $1 \leq i \leq \ell$ by
\[
\left\{
\begin{matrix}
\gamma_i = \zeta_3^{m_i} & \textrm{if $m_i = 1,2$} & \\
\gamma_i=(1,\phi) & \textrm{if $s_i \in S^{(1,\phi)}$} & \\
\gamma_i=(1,\rho) & \textrm{if $s_i \in S^{(1,\rho)}$} & .
\end{matrix}
\right.
\]
Then we define the monodromy vector $\gamma(\Gamma):=(\gamma_1,\dots,\gamma_\ell)$.
\end{itemize}

\subsubsection{Linearization of $\CC^*$--action on $\PP^1$ to a line bundle on $\PP^1$}\label{sect:linearization} $\phantom{a}$ \\
\indent Here we describe linearizations $\tau_{t}$ in details. 
Consider a $\CC^*$--action on $\PP^1$ which fixes two points $p_0 = [0:1]$ and $p_1 = [1:0]$. 
Then, for $t \in \CC^*$, $h(t) \in \Aut(\PP^1)$ should be of the form:
      \[
	  h(t)([x:y]) = [t^k x:y], \ k \in \QQ.  
      \]
For a line bundle $\cO(d)$ over $\PP^1$, consider the linearization $L_t : h(t)_* \cO(d) \to \cO(d) $ given by the following. Let $U_0 = \set{y \neq 0} \subset \PP^1$ be the neighborhood of $p_0$ and let $U_1 = \set{x \neq 0}$ be the neighborhood of $p_1$. Let us represent $\cO_{\PP^1}(d)$ by local trivialization over $U_0$, $U_1$ and a transition function over $U_0\cap U_1$:
\begin{align*}
\psi_{01} : (U_0 \times \CC)|_{U_0\cap U_1} & \to (U_1 \times \CC)|_{U_0 \cap U_1} \\
([x:y],u) & \mapsto ([x:y],(y/x)^d u).
\end{align*}

Then, $h(t)_*\cO(d)$ is given by the following data of local trivialization and a transition function:
\begin{align*}
\psi_{01}' : (U_0 \times \CC)|_{U_0\cap U_1} & \to (U_1 \times \CC)|_{U_0 \cap U_1} \\
([x:y],u) & \mapsto ([x:y],(y/t^{-k} x)^d u).
\end{align*}
Then, a linearization $L_t : h(t)_* \cO(d) \to \cO(d)$ is represented as follows, by using local trivialization, which commutes with transition functions $\psi_{01}, \psi_{01}'$ as follows
\begin{align*}
\xymatrix{
\textrm{Over $U_0$ :} & ([x:y],u) \ar@{|->}[r]^-{\times t^\lambda} \ar@{|->}[d]_-{\psi_{01}'} & ([x:y],t^\lambda u) \ar@{|->}[d]^-{\psi_{01}} \\
\textrm{Over $U_1$ :} & ([x:y],(y/t^{-k} x)^d u) \ar@{|->}[r]_-{\times t^{\lambda - kd}} & t^\lambda(y/x)^d u 
}
\end{align*}
where the points on the left column represent the same point of the line bundle $h(t)_* \cO(d)$ and the points on the right column represent the same point of the line bundle $\cO(d)$, and horizontal arrows represent the morphism $L_t : t^*\cO(d) \to \cO(d)$. Here the weight of $\CC^*$--action on $\cO(d)|_{p_0}$ is $\lambda$, and the weight of $\CC^*$--action on $\cO(d)|_{p_1}$ is $\lambda - kd$.
Therefore, we obtain the following lemma.

\begin{lemm}\label{lemm:linwt}
Let $w_i$ be the weight of the $\CC^*$--representation on $\cO(d)|_{p_i}$ for $i=0,1$. Then we have
\[
(w_0-w_1)/d = k.
\]  
\end{lemm}

\begin{rema}\label{rema:invsec}
For a section $s = x^m y^{d-m} \in \rH^0(\PP^1, \cO(d))$, we can check $L_t(h(t)_* s) = t^{-km+\lambda} s$. Therefore, $s$ is invariant if and only if $\lambda-km = 0$.
\end{rema}

\smallskip

\subsubsection{Linearization of $\CC^*$--action on $\PP(1,3)$ to a line bundle on $\PP(1,3)$}\label{sect:linearization2} $\phantom{a}$ \\
\indent Here we analyse linearizations $\tau_{t,-}$ in details. This will be useful when we describe weights of line bundles at special points(nodes or marked points) of $\cC_{ij}$. Let $[x:y]$ be the coordinate of $\PP(1,3)$. Let $x$ have degree $3$ and $y$ have degree $1$. Let us denote $p_0 = [0:1]$ and $p_1 = [1:0]$. Note that $p_1 \cong [pt/\mu_3]$ is the orbifold point.

\smallskip

Let us consider the morphism $u : \PP^1=\set{[z:y]} \to \PP(1,3)$ given by $u([z:y]) = [z^3:y]$. We denote $[0:1]\subset \PP^1$ by $p_0'$ and denote $[1:0] \in \PP^1$ by $p_1'$. Note that $p_0',p_1'$ are the ramification points, with ramification degree $3$. Consider a $\CC^*$--action on $\PP(1,3)$ which fixes two points $p_0,p_1$. This will correspond to the (coarse space) of $\cC_{ij}$ later. For $t \in \CC^*$, the $h(t) \in \Aut(\PP(1,3))$ should be of the form:
      \[
	  h(t)([x:y]) = [x:t^{-k}y] = [t^{3k}x:y], \ k \in \QQ.  
      \]

For the line bundle $\cO(d)$ over $\PP(1,3)$ and a linearization $L_t : h(t)_* \cO(d) \to \cO(d)$ we consider their pull-backs over $\PP^1$ via the $u$. Then we have $u^* \cO_{\PP(1,3)}(d) \cong \cO_{\PP^1}(3d)$, and we have the induced automorphism $h(t)' \in \Aut(\PP^1)$ of the form: 
\[
h'(t)([z:y]) = [t^k z : y], k \in \QQ.
\]      
Also we have the induced linearization $L_t' : h'(t)_* \cO(3d) \to \cO(3d)$. Then, by computations in Section \ref{sect:linearization}, we have
\[
(w_0' - w_1')/3d = k
\]
where $w_i'$ are the weights of $\CC^*$--representations on $\cO(3d)|_{p_i'}$ for $i=0,1$.

Let $w_i$ be the weight of the $\CC^*$--representation on $\cO(d)|_{p_i}$ for $i=0,1$. Then we have $w_i = w_i'$ for $i=0,1$. Note that $\deg(\cO(d)) = 3d$ here. So that we have the following.

\begin{lemm}\label{lemm:linwt2}
We have \\ 
$(w_0 - w_1)/\deg(\cO(d)) = (w_0 - w_1)/3d = k$ and $T_{\PP(1,3)}|_{p_0} \cong (T_{\PP^1}|_{p_0'})^{\otimes 3} \cong \bL_{3k}, \ T_{\PP(1,3)}|_{p_1} \cong T_{\PP^1}|_{p_1'} \cong \bL_{-k}$ where $\bL_k$ is the representation of $\CC^*$ on $\CC$ with a weight $k \in \ZZ$.
\end{lemm}

\smallskip

\subsubsection{Description of MSP fields over subcurves}\label{sect:MSPdescript}
$\phantom{a}$
\\
\indent From now on, for an arbitrary section $s$ of a line bundle, we denote $s \equiv 0$ if $s$ is the zero section, and $s \equiv 1$ when $s$ is a nonvanishing section. For an MSP field $\xi \in \cW^T$, we investigate the restriction of MSP fields to $\cC_v$, $\cC_e$ as follows. 

\begin{rema}
An MSP field $\xi$ is stable if and only if the restrictions $\xi|_{\cC_i}$, $\xi|_{\cC_{ij}}$ are all stable.
\end{rema}

\begin{conv}
By abuse of notation, we will denote $\xi|_{\cC_v}$ by $(\cC_v,\Sigma^{\cC_v},\cL_1,\cL_2,\phi,\theta,\rho)$. Similar convention for $\xi|_{\cC_e}$.
\end{conv}

\begin{itemize}
\item[(a)] 
\textbf{($0$--vertex: $K_{\mathbb P^2\times \mathbb P^2}$--theory)}
    In this case, $\xi|_{\cC_v}$ satisfies $\mu\equiv 0$ and
    $\nu\equiv 1$. 

Since $\nu \equiv 1$, we have $\cN \cong \cO_{\cC_v}$ and therefore the stability condition (2) becomes

\begin{align}\label{eq:stab0}
\deg(\cL_1 \otimes \cL_2 \otimes (\omega_{\cC}^\log|_{\cC_v})^{\otimes \left( \frac{1}{3} \right)^+} ) > 0.
\end{align}
Because of the stability condition (1) in Definition \ref{cond:stab}, the triple $(\phi_1,\phi_2,\phi_3)$ is nonvanishing. Therefore, \eqref{eq:stab0} is equivalent to that the induced morphism $(\phi,\theta): \cC_v \to \PP^2 \times \PP^2$ is a stable map. 

\medskip

\noindent \textbf{(The $\mathbb C^*$--weights of line bundles)} $\phantom{a}$ \\
\indent Since $\CC^*$ acts trivially on $\cC_v$, 
and $\nu \equiv 1$ is a $\CC^*$--invariant section, $\cN \cong \cO_{\cC_v}$. Moreover, $\mu \equiv 0$, $\phi = (\phi_1,\phi_2,\phi_3)$ is nonvanishing over $\cC_v$ and $\phi_i$ are $\CC^*$--invariant. Since the induced $\CC^*$--action on $\cC_v$ is trivial, $\cL_1$ has weight $0$. Otherwise it cannot have nonzero invariant sections. Similarly we can show that $\cL_2$ has weight $0$ because $(\theta_1,\theta_2,\theta_3)$ is nowhere vanishing and $\theta_i$ are invariant. Note that when $v$ is a $1$--vertex or $\infty$--vertex, $\cL_2|_{\cC_v}$ has weight $0$ by the same reason.
In summary, $(\cL_1,\cL_2,\cN)$ all have $\CC^*$--weight 0.


\medskip

\item[(b)]
      \textbf{($1$--vertex : $\mathbb P^2$--theory)}

In this case $\xi|_{\cC_v}$ satisfies $\mu \equiv \nu \equiv 1$, $\phi \equiv \rho \equiv 0$. 
Since $\mu\equiv \nu \equiv 1$, we have $\cL_1\cong \cN \cong \cO_{\cC_v}$. Therefore, the second stability condition in Definition \ref{cond:stab} becomes

\begin{align}\label{eq:stab1}
\deg(\cL_2 \otimes (\omega_{\cC}^\log|_{\cC_v})^{\otimes \left( \frac{1}{3} \right)^+} ) > 0.
\end{align}
This is equivalent to that $\theta = (\theta_1,\theta_2,\theta_3)$ induces a stable map $\bar{\theta} : \cC_v \to \PP^2$.

\medskip
      
\noindent \textbf{(The $\mathbb C^*$--weights of line bundles)} $\phantom{a}$ \\
\indent Since $\mu \equiv \nu \equiv 1$, we have $\cN \cong \cO_{\cC_v}$ and $\cL_1 \otimes \bL_1 \cong \cO_{\cC_v}$. Therefore $\cL_1 \cong \cO_{\cC_v} \otimes \bL_{-1}$. $\cL_2$ has weight $0$. Thus $\cL_1$ has $\CC^*$--weight $-1$ and $\cN, \cL_2$ have weight $0$. 


\medskip

\item[(c)]
      \textbf{($\infty$--vertex : $\mathbb P^2$--family of $[\CC^3/\mu_3]$--theory)}

\indent In this case, $\xi|_{\cC_v}$ satisfies $\mu\equiv 1, \nu \equiv 0$.
Since $\nu \equiv 0$, we have $\rho \equiv 1$ from the stability condition (1). Thus we have $\cL_1^{-3} \otimes \cL_2^{-3} \otimes \omega^{\log}|_{\cC_{v}} \cong \cO_{\cC_v}$. Therefore we have $\cL_1^{-1} \cong \cL_2 \otimes (\omega^\log_{\cC}|_{\cC_v})^{-\frac{1}{3}}$. Moreover, since we have $\mu \equiv 1$, we have $\cL_1 \otimes \cN \cong \cO_{\cC_v}$. Therefore the stability condition (2) becomes

\begin{align}\label{eq:stabinf}
\deg ( \cL_2^{\otimes 2} \otimes (\omega^\log_\cC|_{\cC_v})^{0+} ) > 0.
\end{align}
This is equivalent to that $\theta = (\theta_1,\theta_2, \theta_3)$ induces a stable map to $\PP^2$.

      \medskip

\noindent \textbf{(The $\mathbb C^*$--weights of line bundles)} $\phantom{a}$ \\
\indent Since $\mu \equiv 1$, we have $\cL_1 \otimes \cN \otimes \bL_1 \cong \cO_{\cC_v}$. And $\cL_2$ has weight $0$.

\end{itemize}

\medskip

For each edge $e$, we describe restriction of MSP fields over $\cC_e$. In these cases, the coarse moduli spaces of $\cC_{e}$ will turn out to be $\PP^1$. 

\begin{itemize}
\item[(d)]
      \textbf{($0$--$1$)--edge}

      Here $e \in E_{01}(\Gamma_\xi)$ and $\xi|_{\cC_e}$ satisfies $\rho \equiv 0, \nu \equiv 1$. Here $\phi = (\phi_{1}, \phi_{2}, \phi_{3})$, $\phi_i$ are $\CC^*$--invariant. Since the $\CC^*$--action on $\cC_e \cong \PP^1$(since orbifold points are concentrated at $\cC_\infty$) is nontrivial, there is only one(up to scalar multiplication) invariant section of $\cL$ by Remark \ref{rema:invsec}. Therefore, $\phi_1,\phi_2,\phi_3$ are proportional to each other.
      
      \smallskip
      
      By the same reason, $\theta_1,\theta_2,\theta_3$ are proportional to each other. Moreover $(\theta_1,\theta_2,\theta_3)$ is nowhere vanishing and we may assume that $\theta_1$ is nowhere vanishing. Hence we have $\cL_2 \cong \cO_{\cC_e}$. Note that $(\theta_1,\theta_2,\theta_3)$ induces a constant morphism $\cC_{e} \to \PP^2$. For other types of edges( ($1$--$\infty$)--type and ($0$--$\infty$)--type ), this still holds.
      
      \smallskip
      
Since $\nu \equiv 1$, $\cN \cong \cO_{\cC_{e}}$ and $\mu$ is a section of $\cL_1$. Since $(\phi,\mu)$ is nonvanishing, it induces a morphism $\cC_{e} \to \PP^3$. Therefore, at least one of $\set{\phi_1,\phi_2,\phi_3,\mu}$ is a nontrivial section of $\cL_1$, so that $\deg \cL_1 \geq 0$. But, by arguments in (a) and (b), $\phi_i=0$ on $\cC_e \cap \cC_1$ and $\mu=0$ on $\cC_e \cap \cC_0$. Therefore $\cL_1 > 0$. 

Since $\phi_1,\phi_2,\phi_3$ are proportional to each others, its image is a line $L$ in $\PP^3$. Therefore, $(\mu,\phi,\theta)$ induces a branched cover over $L \times \set{pt} \subset \PP^3 \times \PP^2$. The line $L$ is of the form
      \[
        L_{a; b} = \{\, ([t: sa_1: sa_2: sa_3], \ [b_1: b_2:b_3]) \mid [t:s] \in \mathbb
        P^1\} \subset \mathbb P^3\times \mathbb P^2 \, .
      \]
      
      
      \medskip

\noindent \textbf{($\CC^*$--weights on the tangent spaces of $\cC_{e}$ at two endpoints)} $\phantom{a}$ \\
\indent The induced $\CC^*$--action on $\cC_{e}$ fixes two endpoints $p_0 = \cC_{e} \cap \cC_0$ and $p_1 = \cC_{e} \cap \cC_{1}$. Thus we can consider weights of the $\CC^*$--actions at the tangent spaces $T_{\cC_{e}}|_{p_0}$ and $T_{\cC_{e}}|_{p_1}$. For this, it is enough to determine what the induced $\CC^*$--action on $\cC_{e} \cong \PP^1$ is. We may assume that $p_0 = [0:1]$ and $p_1 = [1:0]$. The $\CC^*$--action on $\cC_{e}$ gives us an automorphism $h(t)$ on $\cC_{e}$ for an element $t\in \CC^*$. This should be of the form :     
      \[
	  h(t)([x:y]) = [t^k x:y], \ k \in \QQ.  
      \]    
Let $d_{e} = \deg(\cL_1)$. From the discussion on MSP fields over $\cC_v$ above, we know that $\cL_1|_{p_0}$ has weight $0$ and $\cL_1|_{p_1}$ has weight $-1$. Therefore, by Lemma \ref{lemm:linwt}, we have $k = 1/d_{e}$. Therefore, the $\CC^*$--action is of the form $t \cdot [x:y] = [t^{-1/d}x:y]$
      \[
T_{\cC_{e}}|_{p_0} \cong \bL_{\frac{1}{d_{e}}}, \ T_{\cC_{e}}|_{p_1} \cong T_{\cC_{e}}|_{p_0}^{-1} \cong \bL_{\frac{-1}{d_{e}}}.      
      \]

\noindent From the stability condition (2) of Definition \ref{cond:stab}, we have 
$$\deg \cL_1 + \deg (\omega_\cC^\log)|_{\cC_{e}}^{\otimes \left( \frac{1}{3} \right)+} > 0. $$ 
Since $\cC_{e}$ has at most $2$ special points, we have $d_{e} > 0$.
      

\medskip

\item[(e)]
      \textbf{($1$--$\infty$)--edge}

      Here $e \in E_{1\infty}(\Gamma_\xi)$ and $\xi|_{\cC_e}$ satisfies $\mu\equiv 1, \phi \equiv 0$, 
      $\cL_1 \otimes \cN \cong \bL_{-1}$. For the same reason as above, $\cL_2 \cong \cO_{\cC_e}$.
      Furthermore,
      \[
        \begin{aligned}
          & \rho\in \rH^0(\cL_1^{-3} \otimes \cL_2^{-3} \otimes
          \omega_{\cC}^{\log}) = \rH^0(\cL_1^{-3} \otimes \mathcal
            \cO_{\mathbb P^1}(-3)),\\
          & \nu \in \rH^0(\cN) = \rH^0(\cL_1^{-1} \otimes \bL_{-1})
        \end{aligned}
      \]
      gives a branched cover of a weighted projective space $\mathbb P(1,3)$.

\medskip      
\noindent \textbf{($\CC^*$--weights on the tangent spaces of $\cC_{e}$ at two endpoints)} $\phantom{a}$ \\
\indent Let $p_1$ and $p_\infty$ be the two points $\cC_e \cap \cC_1$ and $\cC_e \cap \cC_\infty$ respectively. Here we need to consider cases (1) $p_\infty$ is a scheme point; (2) $p_\infty$ is an orbifold point. We first consider the case (1). We have $\cC_{e} \cong \PP^1$ and we may assume that $p_1 = [0:1]$ and $p_\infty = [1:0]$. 
From the observations over $1$--vertex and $\infty$--vertex, we have $\cL_1|_{p_1} \cong \bL_{-1}$ and $(\cL_1^{-3} \otimes \omega^\log_{\cC}|_{\cC_{e}} )|_{p_\infty} \cong \bL_0$. Let $w_\infty$ be the weight of the $\CC^*$--action over $\cL_1|_{p_\infty}$. Recall that the induced $\CC^*$--action is given by $t \cdot [x:y] = [t^k x : y]$. By Lemma \ref{lemm:linwt}, we have $k = (-1-w_\infty)/d_{e}$, where $d_{e} := \deg\cL_1|_{\cC_e}$. Moreover, we have $T_{\cC_{e}}|_{p_\infty} \cong T_{\cC_{e}}|_{p_1}^{-1} \cong \bL_{-k}$. 

Here there are two subcases, (1--1) $p_\infty$ is not a special point, (1--2) $p_\infty$ is a special point. In case (1--1), so that $p_\infty$ is not a special point. Then we have $(\omega^\log_{\cC}|_{\cC_{e}})|_{p_\infty} \cong \bL_{k}$. Since $\cL_1^{-3} \otimes \omega^\log_{\cC_{e}}|_{p_\infty} \cong \bL_0$, we have $3 + 3kd_{e} + k=0$. Hence we have $k = \frac{-3}{3d_{e} +1 }$, and we have
\[
T_{\cC_{e}}|_{p_1} \cong \bL_{\frac{-3}{3d_{e} + 1}}, \  T_{\cC_{e}}|_{p_\infty} \cong T_{\cC_{e}}|_{p_1}^{-1} \cong \bL_{\frac{3}{3d_{e} + 1}}.
\]
On the other hand, we have $\omega_{\infty} = -1 - k d_{e} = \frac{-1}{3d_{e} + 1}$ and $\cL_{p_\infty} \cong \bL_{\frac{-1}{3d_{e} + 1}}$. 


\smallskip

In case (1--2), $p_\infty$ is a special point. Since $p_\infty$ is a scheme point, we have $\cC_{e} \cong \PP^1$ and we may assume that $p_1 = [0:1]$ and $p_\infty = [1:0]$. Let $z$ be a local coordinate around $p_\infty \subset \cC_{\infty}$, then $dz/z$ generates $\omega^\log|_{\cC_{\infty}}$ around $p_\infty$. Since $dz/z$ is $\CC^*$--invariant, we have $(\omega^\log_{\cC}|_{\cC_{e}})|_{p_\infty} \cong \bL_0$. Then Lemma \ref{lemm:linwt} says that $k = 1/d_{e}$. So that we have
\[
T_{\cC_{e}}|_{p_1} \cong \bL_{\frac{-1}{d_{e}}}, \  T_{\cC_{e}}|_{p_\infty} \cong T_{\cC_{e}}|_{p_1}^{-1}.
\]
Next we consider the case (2), when $p_\infty$ is an orbifold point. Then we have $\cC_{e} \cong \PP(1,3)$. Using the arguments for cases (1-1), (1-2) combined with Lemma \ref{lemm:linwt2}, we have the following.

In case (2-1), $p_\infty$ is not a special point, we have
\[
T_{\cC_{e}}|_{p_1} \cong \bL_{\frac{-3}{3d_{e} + 1}}, \  T_{\cC_{e}}|_{p_\infty} \cong \bL_{\frac{1}{3d_{e} + 1}}.
\]


In case (2-2), $p_\infty$ is a special point, we have
\[
T_{\cC_{e}}|_{p_1} \cong \bL_{\frac{-1}{d_{e}}}, \  T_{\cC_{e}}|_{p_\infty} \cong \bL_{\frac{1}{3d_{e}}}.
\]

\noindent From the stability condition (2) of Definition \ref{cond:stab}, we have 
$$\deg \cL_1^{-1} + \deg (\omega_\cC^\log)|_{\cC_{e}}^{\otimes \left( \frac{1}{3} \right)+} > 0.$$ 
Since $\cC_{e}$ has at most $2$ special points, we have $d_{e} < 0$.

\medskip

\item[(f)]
      \textbf{($0$--$\infty$)--edge}
      
      Here $e \in E_{0\infty}(\Gamma_\xi)$ and $\xi|_{\cC_e}$ satisfies $\rho \neq 0$, $\phi=(\phi_1,\phi_2,\phi_3) \neq 0$. 
We also note that the two fields $(\mu, \nu)$ give a branched cover of some $\mathbb P^1$. 

\medskip

\noindent \textbf{($\CC^*$--weights on the tangent spaces of $\cC_{e}$ at two endpoints)} $\phantom{a}$ \\
\indent Let $p_0$ and $p_\infty$ be the two points $\cC_e \cap \cC_0$ and $\cC_e \cap \cC_\infty$ respectively. We have $\cN|_{p_0}$ has $\CC^*$--weight $0$. Moreover, since $(\cL_1 \otimes \cN)|_{p_\infty} \otimes \bL_1 \cong \bL_0$ by the result for $\infty$--vertex, we have $\cN|_{p_\infty} \cong \bL_{-1}$. When we apply Lemma \ref{lemm:linwt}, Lemma \ref{lemm:linwt2} to $\cN$, we can compute tangent spaces at $p_0,p_\infty$ as follows. We will consider two cases, when $p_{\infty}$ is a scheme point or an orbifold point. Recall that $\deg(\cN) =d_{e \infty}$.
      
\smallskip

\noindent Case (1): $p_\infty$ is a scheme point. In this case, we have $\cC_{e} \cong \PP^1$. We have
\[
T_{\cC_{e}}|_{p_0} \cong \bL_{\frac{1}{d_{e\infty}}}, \  T_{\cC_{e}}|_{p_\infty} \cong \bL_{\frac{-1}{d_{e\infty}}}. 
\]

\smallskip

\noindent Case (2): $p_\infty$ is an orbifold point. In this case, we have $\cC_{e} \cong \PP(1,3)$. We have
\[
T_{\cC_{e}}|_{p_0} \cong \bL_{\frac{1}{d_{e\infty}}}, \  T_{\cC_{e}}|_{p_\infty} \cong \bL_{\frac{-1}{3d_{e\infty}}}. 
\]

\noindent From the stability condition (2) of Definition \ref{cond:stab}, we have 
$$\deg \cN^2 + \deg (\omega_\cC^\log)|_{\cC_{e}}^{\otimes \left( \frac{1}{3} \right)+} > 0.$$ 
Since $\cC_{e}$ has at most $2$ special points, we have $d_{e\infty} > 0$.


\end{itemize}




\subsection{Flat graphs and decomposition of the fixed loci}
$\phantom{a}$

  First we will introduce flat graphs in \cite{CLLL, CL20van, CGLL} in our case, the MSP fields to $\wtil{X}$. Let $\cW^T \subset \cW$ be the fixed locus of the $\CC^*$--action and consider an MSP field $\set{\xi} \in \cW^T$. Let $\Gamma_\xi$ be the localization graph associated to $\xi$.
  Let $q \in \cC(\xi)$ be a separating node, such that $\cC(\xi) = \cC_1 \cup \cC_2$ and $q = \cC_1 \cap \cC_2$. Then $q$ is called \textbf{T-balanced} if $T_q \cC_1 \otimes T_q \cC_2 \cong \bL_0$. 
  In the localization graph, $q$ corresponds to an unstable vertex $q \in V^U(\Gamma_\xi)$ and there are two edges $e, e'$ attached to it. Then $q$ is T-balanced if and only if $T_q \cC_e \otimes T_q \cC_{e'} \cong \bL_0$.
  Then we have the following analogue of \cite[Lemma 2.14]{CLLL}, \cite[Lemma 2.6]{CL20van}
\begin{lemm}\label{lemm1}
Let $q$ be an unstable vertex of $\Gamma_\xi$ and let $e,e'$ be edges attached to $q$. Then $q$ is T-balanced if and only if $q\in V_1(\Gamma_\xi)$, $d_{e} + d_{e'} = 0$, and 
$( \cC_e \cup \cC_{e'} ) \cap \cC_\infty $ is a node or a marking of $\cC$.
\end{lemm}
\begin{proof} 
First, when $q$ is at level $0$ or level $\infty$, then from the weight computation of tangent spaces in Section \ref{sect:MSPdescript}, we can see $\CC^*$--weights of $T_q \cC_e$ and $T_q \cC_{e'}$ have the same sign. Therefore, it cannot be T-balanced. When $q$ is a level $1$ vertex, the only possible case that $q$ is T-balanced is that $e$ is a ($0$--$1$)--edge and $e'$ is a ($1$--$\infty$)--edge. By results in Section \ref{sect:MSPdescript} we have the following.
When $\cC_{e'} \cap \cC_{\infty}$ is not a special point, we have
\[
T_q \cC_e \otimes T_q \cC_{e'} \cong \bL_{\frac{-1}{d_{e}} + \frac{-3}{3d_{e'} + 1} }.
\]
Therefore, $q$ is T-balanced if and only if $3d_{e} + 1 + 3d_{e'} = 0$.
On the other hand, since $ \cC_{e'} \cap \cC_\infty$ is not a special point, $\cC_e$ and $\cC_{e'}$ are both schemes. Therefore $d_{e}, d_{e'} \in \ZZ$ and $3d_{e} + 1 + 3d_{e'} = 0$ cannot be satisfied.
When $\cC_e' \cap \cC_{\infty}$ is a special point, we have
\[
T_q \cC_e \otimes T_q \cC_{e'} \cong \bL_{\frac{-1}{d_{e}} + \frac{-1}{d_{e'}} }. 
\]
Therefore, $q$ is T-balanced if and only if $d_{e} + d_{e'} = 0$.

\end{proof}

\begin{defi}
Let $G_{g,\gamma, \bold{d} }$ be the set of decorated graphs $\Gamma$ such that $(g(\Gamma), \gamma(\Gamma), \bold{d}(\Gamma)) = (g,\gamma,\bold{d})$. A graph $\Gamma \in G_{g,\gamma, \bold{d}}$ is called flat if all unstable vertices are $T$-unbalanced. Here we use the result of Lemma \ref{lemm1} as a definition of $T$-balanced vertices.
\end{defi}

Let $(G_{g,\gamma, \bold{d}})^{\fl}$ be the set of flat graphs. Same as in \cite[Section 2.4]{CLLL}, from a decorated graph $\Gamma = \Gamma_\xi$, we obtain a flat graph $\Gamma^\fl$ by \textbf{`flattening'} the graph $\Gamma$. Consider a decorated graph $\Gamma \in G_{g,\gamma,\bold{d}}$ and a balanced vertex $v \in V_1(\Gamma)$. Then $v$ has two edges attached to it, say $e \in E_{1\infty}(\Gamma)$ and $e' \in E_{01}(\Gamma)$. Let $v_\infty$, $v_0$ be the endpoints of $e,e'$ different from $v$. Flattening of the balanced vertex $v$ is that we eliminate the vertex $v$ and the edge $e,e'$ from the graph $\Gamma$ and add a new edge $\wtil{e} \in E_{0\infty}$ connecting the vertices $v_0$, $v_\infty$. We decorate the new edge $\wtil{e}$ by the degree vector $\bold{d}_{\wtil{e}} = (d_{e \infty},d_{e \infty })$. The flattened graph $\Gamma^\fl$ is obtained by flattening all balanced vertices of $\Gamma$. Then we have the following decomposition of fixed loci, as a direct analogue of the decomposition result in \cite{CGLL}.

\begin{prop}
$$\cW_{g,\gamma,\bold{d}}^T = \bigcup_{\Gamma \in (G_{g,\gamma,\bold{d}})^\fl } \cW_{\Gamma} $$
where $\cW_{\Gamma}$ is a locus of MSP fields $\xi$ such that $(\Gamma_\xi)^\fl = \Gamma$. 
\end{prop}

\medskip

\subsection{Regular and irregular graphs}




\begin{defi}[Regular and irregular graphs]
Let $\Gamma \in (G_{g,\gamma, \bold{d}} )^\fl$ be a flat graph. It is called regular if $E_{0\infty}(\Gamma)=\emptyset$. It is called irregular if it is not regular.
\end{defi}

\begin{rema}
The above definition of irregular graph differs slightly from the original one in \cite{CLLL, CL20van} which considered the quintic 3-fold case. In the original definition, even if $E_{0 \infty}(\Gamma)= \varnothing$, $\Gamma$ can be irregular if the monodromy types of legs(corresponding to markings) and monodromy type of flags(corresponding to nodes) do not satisfy certain condition. See \cite[Definition 2.8]{CL20van} for details.

But in this paper, we only show that the contributions corresponding to localization graphs containing $0$--$\infty$ edges vanish.
\end{rema}

We will show the localized virtual cycle $[\cW_\Gamma]\virt\loc$ does not contribute to invariants when $\Gamma$ is an irregular graph and not a pure loop. 
Here \textbf{pure loop} means a graph which has no legs, has no stable vertices, and for every vertex exactly two edges are attached to it.

\section{Proof for irregular vanishing}\label{sect:insepvan}

  In this section, for an irregular graph $\Gamma$ which is not a pure loop, we will show that integrations over the (cosection) localized virtual cycle $[\cW_\Gamma]\virt\loc$ vanish. 
  Recall that our moduli space of MSP fields, $\cW$ has a natural $\CC^*$--invariant perfect obstruction theory. 
   Let $\cW_{\Gamma}^- \subset \cW_{\Gamma}$ be the degeneracy loci of the cosection $\sigma_\gamma : h^1(E_{\cW_{\Gamma}}\dual) \to \cO_{\cW_{\Gamma}}$ which is induced from the cosection $\sigma$. Since $\cW^-$ is proper, $\cW^-_\Gamma$ is also proper. 

For a cycle $A \in A_*(\cW_{\Gamma}^-)$, we call $A\sim 0$ if there exists a proper closed substack $Z \subset \cW_{\Gamma}$ containing $\cW_{\Gamma}^-$ and $j_*A = 0$ for the inclusion $j : \cW_{\Gamma}^- \hra Z$. It is clear that $A \sim 0$ implies integration over $A$ always vanishes. In this section, we will show the following.
\begin{align}\label{eq:insepvan}
\left[ \cW_\Gamma \right]\virt\loc \sim 0.
\end{align}

Note that we can use virtual localization formula \cite{CKL} for cosection localized virtual cycles. In the proof, we will generally follow the strategy in \cite{CL20van}, which showed irregular vanishing on MSP fields for quintic 3-folds. In the following we will show that it is enough to prove the special case where the (decorated) graph $\Gamma$ satisfies that $V_1(\Gamma)  = \emptyset$ and $\Gamma$ has no strings, and no legs $s \in S^{(1,\rho)} \cup S^1$. Here, \textbf{`strings'} are edges $e \in E_{0 \infty}(\Gamma)$, whose vertex $v \in V_0(\Gamma)$ is unstable and there is no other edge attached to $v$.
  The following lemma, which is an analogue of a lemma in \cite{CLLL} is necessary for our later use.
\begin{lemm}\cite[Lemma 2.12]{CLLL} \label{lemm2}
Let $\xi \in \cW^T$ and $e \in E_{0 \infty}(\Gamma_\xi)$. Then $\cC_e \cong \PP^1$, $\cL_1|_{\cC_e} \cong w^{\log}_{\cC}|_{\cC_e} \cong \cO_{ \cC_e}$, and $\cN|_{ \cC_e} \cong \cO_{ \cC_e}(d_{e\infty})$. 
\end{lemm}
\begin{proof} 
By abuse of notation, let $\cL_i = \cL_i|_{\cC_e}$ and $\cN = \cN|_{\cC_e}$ here. From the arguments in Section \ref{sect:MSPdescript} (f), $\phi\neq 0$ and $\rho \neq 0$. Since $\phi \neq 0$, at least one of $\phi_i$ is a nonzero section of $\cL_1$ and we may assume that $\phi_1|_{\cC_e} \neq 0$. Therefore $\deg \cL_1 \geq 0$. Moreover, since $\rho \neq 0$, we have 
$\deg\left( \cL_1^{-3} \otimes \cL_2^{-3} \otimes \omega_\cC^\log|_{\cC_e} \right) \geq 0$. 

Since $\cL_2$ is trivial and $\deg \omega_\cC^\log|_{\cC_e} \leq 0$, we conclude that $\deg \cL_1 = \deg \omega_\cC^\log|_{\cC_e} = 0$. Now we verify that $\cC_e = \PP^1$. Let $\cC_e \cap \cC_0 = p_0$ and $\cC_e \cap \cC_\infty = p_\infty$. Since orbifold points of $\cC$ are all contained in $\cC_\infty$, it is enough to show that $p_\infty$ is a scheme point.
If $p_\infty$ were an orbifold point, then because $\Aut(p_\infty)$ acts on $\cL_1|_{p_\infty}$ faithfully, we have $\phi|_{p_\infty}=0$. Therefore, if $p_\infty$ were an orbifold point, then it is the only orbifold point where $\phi_1=0$. Then $\deg \cL_1$ is not an integer, which leads to a contradiction.

Therefore, from the information of degrees of the line bundles, we obtain $\cC_e \cong \PP^1$, $\cL_1|_{\cC_e} \cong w^{\log}_{\cC}|_{\cC_e} \cong \cO_{\cC_e}$, and $\cN|_{ \cC_e} \cong \cO_{ \cC_e}(d_{e\infty})$.
\end{proof}

\subsection{Vanishing for the special case} \label{sect:nostring}

  Here we will show that $[\cW_\Gamma]\virt\loc = 0$ for the special case that there are no strings, no vertex $v \in V_1$ and no markings decorated by $(m=1)$ or $(1,\rho)$. We can prove it by showing that the virtual dimension is negative. 
  Let $D_\Gamma$ be the moduli space parametrizing $(\cC, \Sigma^{\cC}, \cL_1, \cL_2, \cN)$, such that the localization graph for $\cC$ is equal to $\Gamma$ and let $S_v$ be the set of legs connected to $v$. Then we have
\begin{align}\label{eq:vdimcount1}
\dim D_\Gamma & = \sum_{v\in V^S} (3 g_v - 3 + |E_v| + |S_v| ) + \left( \sum_{v \in V^S} 3g_v \right) + 3 h^1(\Gamma) - |E| - 3.
\end{align}
\noindent Next we consider virtual dimensions from deformations of $\mu, \nu$, which is
$$
\chi_T( \cL_1 \otimes \cN \otimes \bL_1) + \chi_T(\cN), \ \ \chi_T(-):=\sum_i (-1)^i \dim(\rH^i(-)^{\CC^*}).
$$
Similar to \cite[p. 7367]{CL20van}, we can show that it is equal to
\begin{align}\label{eq:vdimcount2}
\sum_{v \in V_0} (1 - g_v) + \sum_{v \in V_\infty} (1- g_v).
\end{align}

Next we consider virtual dimensions from deformations of $\phi, \theta, \rho$. Similar to \cite[(4.4)]{CL20van}, we can show that it is equal to 
$$
3 \chi( \cL_1(-\Sigma^{\cC}_{(1,\phi)}) ) + 3 \chi( \cL_2) + \chi( \cL_1^{-3} \otimes \cL_2^{-3} \otimes \omega_{\cC}^{\log} (-\Sigma^{\cC}_{(1,\rho)})).
$$
Moreover, it is equal to
\begin{align}\label{eq:vdimcount3}
& -3|\Sigma^{\cC}_{(1,\phi)}| + 3\left(\deg \cL_1 + 1 - g - \sum_{a \in S^{\neq 0}} \frac{m_a}{3} \right) + 3(\deg \cL_2 + 1 - g) \\ \nonumber
& + \left( 2g-2 + |S| -3 \deg \cL_1 - 3 \deg \cL_2 - |\Sigma^{\cC}_{(1,\rho)}| + 1 - g \right).
\end{align}
Therefore, we have $\vdim \cW_\Gamma$ is equal to \eqref{eq:vdimcount1} + \eqref{eq:vdimcount2} + \eqref{eq:vdimcount3}, which is equal to

\begin{align}\label{eq:vdim}
\sum_{v \in V^S} |E_v| - 2|\Sigma^{\cC}_{(1,\phi)}| + \sum_{v \in V^S_\infty} |S^0_v| - \sum_{a \in S^{\neq 0}}(m_a - 2) - 3(|E|-|V^U|)
\end{align}
where $V^S$ is the set of stable vertices, and $V^U$ is the set of unstable vertices.

\medskip

Next we focus on special sub-graph of $\Gamma$ called chains. Consider a sequence of edges $E_1, \dots, E_k$ where $E_i$ and $E_{i+1}$ are connected at one vertex. Let $v_{i-1}, v_i$ be the vertices of the $i$-th edge $E_i$. Then we call $\set{E_1,\dots, E_k}$ is a chain if $v_1,\dots, v_{k-1}$ are all unstable vertices. Consider a \textbf{maximal chain} $\set{E_1, \dots, E_k}$. Since it is maximal, at least one of $v_0, v_k$ should be stable. 

\begin{rema}
Since $\Gamma$ is not a pure loop, we can exclude the case that $v_0,v_k$ are both unstable. If $\Gamma$ is a pure loop, we have $V^S = S = \emptyset$ and $|E|=|V^U|$ and therefore the virtual dimension \eqref{eq:vdim} becomes zero.
\end{rema}

\noindent Case 1) Only one of $v_0, v_k$ is stable. Let us assume $v_0$ is stable and $v_k$ is unstable. 
Since there are no strings and $V_1(\Gamma)$ by the assumption, $E_k \in E_{0 \infty}(\Gamma)$ and $v_k \in V_{\infty}(\Gamma)$. 

We will show that $|S_{v_k}^{(1,\phi)}| = 1$. Recall that for any MSP field $\xi \in \cW_\Gamma$, we have $(\Gamma_\xi)^\fl = \Gamma$ where $\Gamma_\xi$ the decorated graph associated to $\xi$ defined in Section \ref{sect:assocgraph}, and $(\Gamma_\xi)^\fl$ is its flattening.   We first consider the case that $\cC_{E_k} \subset \cC_\xi$ is irreducible.(So that $E_k \in \Gamma_\xi$ and does not arise as a flattening of a $T$--balanced node.) Then by Lemma \ref{lemm2}, we have $\cC_{E_k} \cong \PP^1$ and $\omega_{\cC_\xi}^\log|_{\cC_{E_k}} \cong \cO_{\cC_{E_k}}.$ If there is no leg attached at $v_k$ then we have $\omega_{\cC_\xi}|_{\cC_{E_k}} \cong \cO_{\PP^1}(-1)$ since the node corresponding to $v_{k-1} $ is the only special point of $\cC_{E_k}$. Thus we obtain a contradiction. 
  Therefore, exactly one leg(since $v_k$ is unstable) is attached at $v_k$, which means that $\cC_{v_k}$ is a marked point. But since $v_k \in V_\infty$ and $\rho \equiv 1$ on $\cC_{v_k}$, the leg must have monodromy type $(1,\phi)$ and $|S_{v_k}^{(1,\phi)}| = 1$. 

  
  Next, consider the case that $E_k$ arises as a flattening of a $T$--balanced node $q$ in $\Gamma_\xi$. Let $e \in E_{1 \infty}(\Gamma_\xi), e' \in E_{01}(\Gamma_\xi)$ be the edges attached to $q$.
  Then, by Lemma \ref{lemm1} and since $v_k$ is unstable, there must exist exactly one leg attached to $v_k$. From Lemma \ref{lemm1}, we have $d_{e} + d_{e'} = 0$. Since there is no orbifold marking over $\cC_{e'}$, $d_{e'}$ is an integer. Therefore $d_{e}$ is an integer as well. Note that $\cC_{E_k}$ only have one marking corresponding to the unique leg $\ell$ attached to $v_k$(=$\cC_{v_k}$). Since $\rho \equiv 1$ over $\cC_\infty$, the marking must have monodromy type $1$ or $2$ or $(1,\phi)$, so that $\phi$ vanishes over the marking. Since $\phi=(\phi_1,\phi_2,\phi_3) \neq 0$ over $\cC_{E_k}$, the marking should be a scheme marking because $d_{e}(=\deg \cL|_{\cC_{e}})$ is an integer. So the marking must have monodromy type $(1,\phi)$. Thus $|S^{(1,\phi)}_{v_k}|=1$. 
  
  Hence we checked $|S^{(1,\phi)}_{v_k}|=1$. Hence the contribution of $\set{E_1,\dots, E_k, v_1, \dots, v_k}$ in \eqref{eq:vdim} is equal to $1 - 2|S^{(1,\phi)}_{v_k}| = -1$.    
  
\medskip  

\noindent Case 2) $v_0, v_k$ are stable. Then the contribution of $\set{E_1,\dots, E_k, v_1,\dots, v_{k-1} }$ is $2-3 = -1$, which is negative. 

\medskip

Therefore, when we consider a graph $\Gamma'$, which is obtained by removing all edges and unstable vertices, and all legs attached to unstable vertices (in fact, connected to $v_0$ or $v_k$). 
Since $E_{0\infty}(\Gamma) \neq \emptyset$, we have
\[
\vdim \cW_\Gamma < \vdim \cW_{\Gamma'}.
\]
Next, apply \eqref{eq:vdim} to $\Gamma'$. Since there are no edges and unstable vertices in $\Gamma'$, it is a sum of two  kinds of contributions:
\begin{itemize}
\item[i)]
Terms from elements in $\cup_{v \in V^S_\infty} S^0_v \,$: Each contribution is $-2 + 1 = -1$. Note that $\cup_{v \in V^S_\infty} S^0_v = \Sigma^{\cC}_{(1,\phi)}$ since there are no unstable vertices and no $(1,\rho)$ vertices by the assumption.
\item[ii)]
Terms associated to $a \in S^{\neq 1}$: Each contribution is $m_a -2 = 0$, hence it is less than or equal to $0$. 
\end{itemize}



  By i), ii), $\vdim \cW_{\Gamma'} \leq 0$. Thus we have $\vdim \cW_\Gamma < \vdim \cW_{\Gamma'} \leq 0$, hence we have
\begin{align}\label{eq:specialvan}
  [\cW_\Gamma]\virt\loc = 0.
\end{align}   
  
  In the following sections, we will show how we can reduce the proof of the irregular vanishing \eqref{eq:insepvan} to the result of this section.



\subsection{Decoupling and trimming of edges $E_{01}$ and $E_{0 \infty}$} \label{sect:decouptrim}

Let $\Gamma$ be a flat graph without string. We will construct a new graph $\Gamma'$ with $E_{01}(\Gamma') = E_{1 \infty}(\Gamma') = V_1(\Gamma') = \emptyset $ in this section. Furthermore we will construct a flat morphism $\psi : \cW_{\Gamma} \to \cW_{\Gamma'}$ such that $\psi^* [ \cW_{\Gamma'} ]\virt\loc = \cW_\Gamma $. This will enable us to focus on the case $V_1(\Gamma) = \emptyset $ as in Section \ref{sect:nostring}.

\smallskip

  As a first step, we consider the process called \textbf{`decoupling'}.
We define leaf edges in the same manner as in \cite{CLLL}.
  An edge $e \in \Gamma$ is called \textbf{`leaf edge'} if one of its vertices (which we call the connecting vertex) is stable or has valency $2$, and the other vertex (which we call the end vertex) is unstable and has valency $1$.
Consider a non-leaf edge $e \in E_{01} \cup E_{1\infty}$ and let $v_1(e)$ be the vertex of $e$ in $V_1$ and let $v_0$(resp. $v_\infty$) be the vertex in $V_0$(resp. $V_\infty$). We note that the only difference from the quintic Mixed-Spin-P field case is that, separating the node corresponding to the flag $(e, v_1(e))$, when $v_1(e)$ is stable, there is a glueing issue (which will be stated below), which does not appear in quintic case.

\medskip

  We consider a new graph $\Gamma'$ obtained from $\Gamma$ by decoupling the flag $(e,v_1(e))$. Precisely, we first remove edge $e$ and then connect a new edge $e'$ connecting $v_0$(resp. $v_\infty$) and a new vertex $v_1' \in V_1(\Gamma')$. Then we connect new legs $\ell$(resp. $\ell'$) decorated by $(1,\phi)$, at $v_1$(resp. $v_1'$). 
  Then the field $\theta$ gives us a morphism $\cW_{\Gamma'} \to \PP^2 \times \PP^2$ given by evaluating the value of the universal morphism $\bar{\theta} : \cC_{\cW_{\Gamma'}} \to \PP^2$ at the marked points $q_\ell, q_{\ell'}$ corresponding to the legs $\ell, \ell'$ in the universal curve. Precisely, $\ell, \ell'$ are sections from $\cW_{\Gamma'}$ to the universal curve $\cC_{\cW_{\Gamma'}}$.

Also we have a morphism $\cW_{\Gamma} \to \PP^2 $ obtained by evaluating the value of the universal morphism $\bar{\theta} : \cC_{\cW_{\Gamma}} \to \PP^2$ at the node $x$ in the universal curve, which corresponds to the flag $(e,v_1(e))$.
Let $u : \cW_\Gamma \to \cW_{\Gamma'}$ be the natural morphism obtained by separating nodes. Then we observe that there is a fiber diagram :
\begin{align}\label{eq:decoupfiber1}
\xymatrix{
\cW_{\Gamma} \ar[r]^-{u} \ar[d]_-{\ev } \ar@{}[rd]|{\Box} & \cW_{\Gamma'} \ar[d]^-{\ev \times \ev } \\
\PP^2 \ar[r]^-{\Delta} & \PP^2 \times \PP^2.
}
\end{align}
We will show that $\Delta^! [\cW_{\Gamma'}]\virt\loc = [\cW_\Gamma]\virt\loc$, using the virtual pull-back formula \cite{Man12}. It is enough to show that there is a distinguished triangle:
\begin{align}\label{eq:distvpull}
ev^*N_{\Delta}[-1] \to u^*E_{\cW_\Gamma'} \to E_{\cW_\Gamma} \stackrel{+1}{\lra} \cdots
\end{align}

  Let us consider a stack $\cD_\Gamma$ parametrizing partial data of MSP fields $\xi \in \cW_\Gamma$ as follows. We say that $\zeta=(\cC, \Sigma^{\cC}, \cL_1, \cN)$ is $T$--equivariant if it is equipped with a morphism $h : \CC^* \to \Aut(\cC, \Sigma^{\cC})$ and linearizations $\tau_{t,\cL_1} : h(t)_* \cL_1 \to \cL_1$, and $\tau_{t,\cN} : h(t)_* \cN \to \cN$. We can consider associated decorated graph $\Gamma_\zeta$ since datum of $\cL_2$ is not used to define $\Gamma_\xi$ in \ref{sect:assocgraph}. Furthermore, we can observe that for an edge $e \in E(\Gamma_\zeta)$, $\zeta$ defines $\CC^*$--weights on the tangent spaces of $\cC_e$ at two endpoints. Therefore, we can define the notion of $T$--balanced node of $\cC$, so that we can define the flattening $(\Gamma_\zeta)^\fl$ of $\Gamma_\zeta$. Then we define $\cD_\Gamma$ to be a stack parametrizing $T$--equivariant $\zeta=(\cC, \Sigma^{\cC}, \cL_1, \cN)$ such that $(\Gamma_\zeta)^\fl = \Gamma$. As in \cite[p. 264]{CLLL}, it is a smooth Artin stack.
  
  Moreover, let $\what{\cD}_\Gamma$ be an Artin stack parametrizing the data $(\cC, \Sigma^{\cC}, \cL_1, \cN, \mu, \nu)$ which satisfies the following:
\begin{enumerate}
\item $(\cC, \Sigma^{\cC}, \cL_1, \cN) \in \cD_\Gamma$. In particular, it is equipped with $T$--equivariant structure.
\item $\mu \in \rH^0(\cL_1\otimes \cN \otimes \bL_1)^T$, $\nu \in \rH^0(\cN)^T$ such that $\mu|_{\cC_0} = \nu|_{\cC_\infty} = 0$ and $\mu|_{\cC_\infty} \equiv 1$, $\nu|_{\cC_0} \equiv 1$, $\mu|_{\cC_1} \equiv 1$, $\nu|_{\cC_1} \equiv 1$.
\item 
For an edge $e \in E_{0\infty}(\Gamma)$ which comes from flattening, $\mu|_{\cC_{1\infty}} \equiv 1$ and $\nu|_{\cC_{01}} \equiv 1$.
\end{enumerate}

This parametrizes partial data of $\CC^*$--fixed MSP fields, forgetting $\vphi,\rho$ fields. Note that this definition of $\what{\cD}_\Gamma$ is parallel to \cite[Definition 3.2]{CL20van} and we can show that it is a smooth Artin stack in the same manner as in \cite[Lemma 3.3]{CL20van}.

Note that the moduli space $\cD_\Gamma$ is different from the moduli space $D_\Gamma$ defined in the beginning of Section \ref{sect:nostring}. Let $p : \cW_\Gamma \to \what{\cD}_\Gamma$, $p' : \cW_{\Gamma'} \to \what{\cD}_{\Gamma'}$ be the forgetful morphisms.

Let $E_{\cW_\Gamma} \to L_{\cW_\Gamma}$, $E_{\cW_\Gamma / \what{\cD}_\Gamma } \to L_{\cW_\Gamma / \what{\cD}_\Gamma}$  be the absolute and relative perfect obstruction theories of $\cW_{\Gamma}$, respectively, which are defined analogously as in \cite{CLLL, CL20van}. Moreover, let $\pi : \cC_{\cW_\Gamma} \to \cW_\Gamma$ be the universal curve then we can observe that
\begin{align}\label{eq:relobs1}
& E_{ \cW_{\Gamma} / \what{\cD}_\Gamma }  = R\pi^T_* \left( \bar{\theta}^* T_{\PP^2} \oplus \cL_1^{\oplus 3} \oplus \left( \cL_1^{-3} \otimes \bar{\theta}^* \cO_{\PP^2}(-3) \otimes \omega_{\cC_{\cW_\Gamma} / \cW_\Gamma }(-\Sigma^\cC_{(1,\rho)}) \right) \oplus (\cL \otimes \cN \otimes \bL_1) \oplus \cN \right).
\end{align} 
Here $R \pi_*^T(-)$ is defined as follows. Since $\pi$ is a $T$--equivariant morphism, this induces a morphism $\bar{\pi} : [\cC_{\cW_\Gamma}/T] \to [\cW_\Gamma / T]$. Thus we have a functor $R \bar{\pi}_* : D^b([\cC_{\cW_\Gamma}/T]) \to D^b([\cW_\Gamma / T])$. Then we define $R \pi_*^T := Lq^* \circ R \bar{\pi}_*$ where $q : \cW_\Gamma \to [\cW_\Gamma / \Gamma]$ is the quotient map. We also note that in \eqref{eq:relobs1} we apply $R \pi_*^T$ to $T$--equivariant bundles, which can be considered as bundles over $[\cW_\Gamma / T]$. 

By \cite[Lemma 2.6]{May01}, we obtain the following diagram:
\begin{align}\label{diag:obsdecoup}
\xymatrix{
A \ar[r] \ar[d] & E_{\cW_{\Gamma} / \what{\cD}_\Gamma } \ar[r] \ar[d] & u^* E_{\cW_{\Gamma'} / \what{\cD}_{\Gamma'}} \ar[r]^-{+1} \ar[d] & \cdots \\
E_{\cW_\Gamma / \cW_{\Gamma'} } \ar[r] \ar[d] & E_{\cW_{\Gamma}} \ar[r] \ar[d] & u^* E_{\cW_{\Gamma'}} \ar[r]^-{+1} \ar[d] & \cdots \\
p^* T_{\what{\cD}_{\Gamma} / \what{\cD}_{\Gamma'} } \ar[r] \ar[d]^-{+1}  & p^* T_{ \what{\cD}_\Gamma} \ar[r] \ar[d]^-{+1} & u^* (p')^* T_{ \what{\cD}_{\Gamma'}} \ar[r]^-{+1} \ar[d]^-{+1} & \cdots \\
\vdots & \vdots & \vdots & &
}
\end{align}
where the rows and columns are distinguished triangles. Note that the relative obstruction theories appears in the above diagrams, and morphisms between them are obtained by taking mapping cones of the bottom-right square. 

Thus we have $A \cong Cone(E_{\cW_\Gamma / \what{\cD}_\Gamma} \to u^* E_{\cW_{\Gamma'} / \what{\cD}_{\Gamma'}})[-1]$. Then $A$ fits into the diagram \eqref{diag:obsdecoup} by the octahedral axiom of the derived category. 
  Then, from \cite[(3.14),(3.17)]{CLLL} and \cite[p.608]{Beh97}, we have 
\begin{align}\label{eq:obscompare1}
A = & R\pi_*^T \left( \left. \left( \cL_1(-\Sigma^{\cC}_\phi)^{\oplus 5} \ \oplus \ \cL_1^{-3} \otimes \bar{\theta}^* \cO_{\PP^2}(-3) \otimes \omega_{\cC_{\cW_\Gamma} / \cW_\Gamma }(-\Sigma^{\cC}_{(1,\rho)}) \right) \right|_x \right) [-1]
\\ 
& \oplus ev^* N_\Delta [-1] \nonumber
\end{align}
where $x : \cW_{\Gamma} \to \cC_{\cW_{\Gamma}}$ is the section into the node. Since the torus $\CC^*$ acts on $\cL_1|_x$ non-trivially, we have 
$$
R\pi_*^T\left( \left. \left( \cL_1(-\Sigma^C_\phi)^{\oplus 5} \right) \right|_x \right) = R\pi_*^T\left( \left. \left(  \cL_1^{-3} \otimes \bar{\theta}^* \cO_{\PP^2}(-3) \otimes \omega_{\cC_{\cW_\Gamma} / \cW_\Gamma }(-\Sigma^\cC_{(1,\rho)}) \right) \right|_x \right) = 0.
$$
Therefore we have 
$$
A \cong ev^* N_\Delta [-1].
$$

\begin{lemm}
The node-splitting morphism $\what{\cD}_\Gamma \to \what{\cD}_{\Gamma'}$ is an isomorphism. Therefore we have $p^*T_{\what{\cD}_\Gamma / \what{\cD}_{\Gamma'}} = 0$.
\end{lemm}
\begin{proof}
First consider the case when $\Gamma'$ is connected. Then by \cite[Proof of Lemma 3.9]{CLLL} the morphism $\cD_\Gamma \to \cD_{\Gamma'}$ is a $(\CC^*)^2$-torsor. The first and the second factor $\CC^*$ of $\CC^*$ are gluing data of $\cN|_{\cC_{e'}}$ and $\cN|_{\cC_{v_1'}}$ at the markings $x(\ell),x(\ell')$, and the gluing data of $\cL_1|_{\cC_{e'}}$ and $\cL_1|_{\cC_{v_1'}}$ at the markings $x(\ell),x(\ell')$.
But $\what{\cD}_{\Gamma'}$ has the data of sections $\mu,\nu$, such that both are nonzero at the markings $x(\ell),x(\ell')$. Therefore, the sections give us a unique way of gluing. Therefore there is no $(\CC^*)$--gerbe structure here.
When $\Gamma'$ has many connected components, we can prove this in a similar manner.
\end{proof}
Thus, the distinguished triangle of the leftmost column of 
\eqref{diag:obsdecoup} gives us an isomorphism
\[
E_{\cW_{\Gamma} / \cW_{\Gamma'}} \simeq ev^* N_\Delta [-1].
\]

  Therefore we obtain the distinguished triangle \eqref{eq:distvpull} we wanted, which proves $\Delta^![\cW_{\Gamma'}]\virt = [\cW_\Gamma]\virt$. Next we consider cosection localized virtual cycles.
  Recall that the cosection $h^1(E_{\cW}\dual) \to \cO_{\cW}$ is induced from the potential $F(\bx, \by)\cdot p$.
As in \cite[Proof of Lemma 3.9]{CLLL}, we can check two cosections $\sigma_\Gamma : h^1(\EE_{\cW_\Gamma}\dual) \to \cO_{\cW_\Gamma}$ and $\sigma_{\Gamma'} : h^1(\EE_{\cW_\Gamma}\dual) \to \cO_{\cW_\Gamma}$ are consistent. Moreover the following diagram commutes:
\[
\xymatrix{0 \ar[r] & h^1(E_{\cW_{\Gamma}}\dual) \ar[rd]_-{\sigma_{\Gamma}} \ar[r] & u^*h^1(E_{\cW_{\Gamma'}}\dual) \ar[r] \ar[d]_-{\sigma_{\Gamma}} & ev^*N_{\Delta} \ar[ld]^-{0} \ar[r] & 0  \\
& & \cO_{\cW_\Gamma} & & }
\]
Therefore, by using virtual pull-backs for cosection localized virtual cycles \cite{Man12, CKL}, we also obtain the equality of localized virtual cycles 
\begin{align}\label{eq:virtpb1}
\Delta^![\cW_{\Gamma'}]\virt\loc = [\cW_\Gamma]\virt\loc.
\end{align}
Therefore it is enough to show $[\cW_{\Gamma'}]\virt\loc = 0$ in order to show $[\cW_\Gamma]\virt\loc = 0$.

\medskip

Now we only need to consider ``decoupled" graphs to show vanishing of the localized virtual cycle $[\cW_\Gamma]\virt\loc$. After decoupling, we proceed with the process called \textbf{`trimming'}, that removes all edges in $E_{01}\cup E_{1\infty}$. For $e \in E_{01}$ such that its incident vertex $v_0(e) \in V_0$ is stable, or $\val(v_0(e)) \geq 2$ we remove this edge $e$ and the vertex $v_1(e) \in V_1$ incident to $e$, and all legs attached to $v_1(e)$, and then we attach a new $(1,\rho)$--leg to $v_0(e)$. Similarly, for $e \in E_{1\infty}$ such that its incident vertex $v_\infty(e) \in V_\infty$ is stable, or $\val(v_\infty(e)) \geq 2$ we remove this edge $e$ and the vertex $v_1(e) \in V_1$ incident to $e$, and all legs attached to $v_1(e)$, and then we attach a new $(1,\phi)$--leg to $v_\infty(e)$. 

  Let $\Gamma'$ be a graph obtained from $\Gamma$ by decoupling all non-leaf edges $e \in E_{01} \cup E_{1\infty}$. Also we consider $\Gamma''$ which is obtained from $\Gamma'$ by trimming all edges in $E_{01}(\Gamma') \cup E_{1\infty}(\Gamma')$.
  Then there is a morphism $q : \cW_{\Gamma'} \to \cW_{\Gamma''}$. By using analogous arguments in \cite{CLLL}, we can show that it is a $G$--gerbe, for a finite group $G$.  
  
\begin{lemm}
$q : \cW_{\Gamma'} \to \cW_{\Gamma''}$ is a $G$--gerbe for a finite group $G$.
\end{lemm}
\begin{proof}
This is straightforward from the proofs of \cite[Proposition 2.32]{CLLL}. The only difference is that we have the additional data $(\cL_2,\theta)$. But, for an edge $e$ we trim from $\Gamma'$, the induced morphism $\bar{\theta} : \cC_e \to \PP^2$ is a constant map by arguments in Section \ref{sect:MSPdescript}, and its value is determined by its value at the connecting node, $\cC_e \cap \cC_\infty$. Therefore, the data $(\cL_2,\theta)$ is fixed over $\cC_e$ so that we can omit it. 
\end{proof}


Moreover, as a direct analogue of \cite[Lemma 3.11]{CLLL}, we have
\begin{align}\label{eq:virtpb2}
q^*[\cW_{\Gamma'}]\virt\loc = [\cW_{\Gamma''}]\virt\loc.
\end{align}

Combining \eqref{eq:virtpb1} and \eqref{eq:virtpb2}, it is enough to show that $[\cW_{\Gamma''}]\virt\loc = 0$ in order to show that $[\cW_{\Gamma}]\virt\loc = 0$. Then, there are no edges between $V_1(\Gamma'')$ and $V_0(\Gamma'') \cup V_\infty(\Gamma'')$. Let $\Gamma_{0\infty}'' \subset \Gamma''$ be a subgraph obtained by removing all $V_1(\Gamma'')$ vertices, and let $\Gamma_1''$ be a subgraph obtained by removing all $V_0(\Gamma'') \cup V_{\infty}(\Gamma'')$. Then the vertex, edge, leg sets of $\Gamma''$ are the disjoint unions of the corresponding sets of the two subgraphs $\Gamma_{0\infty}''$ and $\Gamma_{1}''$ respectively.

Since the original graph $\Gamma$ is irregular, $\Gamma_{0\infty}''$ is nonempty. Obviously we have $\cW_{\Gamma''} = \cW_{\Gamma_{0\infty}''} \times \cW_{\Gamma_{1}''}$ and moreover we have
\begin{align}\label{eq:virtloc3}
[\cW_{\Gamma''}]\virt\loc = [\cW_{\Gamma_{0\infty}''}]\virt\loc \times [\cW_{\Gamma_1''}]\virt\loc \, .
\end{align}

Therefore, to show that $[\cW_{\Gamma''}]\virt\loc = 0$, it is enough to show that $[\cW_{\Gamma_{0\infty}''}]\virt\loc = 0$. Therefore, we may assume that $\Gamma$ has no $V_1(\Gamma)$ vertices.

\subsection{Forgetting $(1,\rho)$ and (m=1) markings} \label{sect:forget} Here we will assume that $V_1(\Gamma) = \emptyset$. For a leg $s \in S^1 \cup S^{(1,\rho)}$, we consider a graph $\Gamma'$ obtained from $\Gamma$ by forgetting $s$ and stabilizing(if necessary).
Then, as a straightforward analogue of the same process in \cite[pp. 7369--7371]{CL20van}, and by \cite[Theorem 4.5]{CLL15} we can check that there is a flat forgetful morphism $f : \cW_\Gamma \to \cW_{\Gamma'}$ and we have $f^*[\cW_{\Gamma'}]\virt\loc = [\cW_{\Gamma}]\virt\loc$. Therefore, to show that $[\cW_{\Gamma}]\virt\loc = 0$, it is enough to show this for the case that $\Gamma$ does not have legs $s \in S^1 \cup S^{(1,\rho)}$. 

\subsection{Reduction to no string case}\label{sect:redtonost} Now we show why we can reduce the situation to the case without strings. Let $e$ be a string of $\Gamma$ with two vertices $v_1, v_2$. Let $\bar{e}$ be a graph with one edge $e \in E_{0 \infty}$ and two vertices $v_\infty(e), v_0(e)$ in $V_\infty$ and $V_0$ respectively, and with a leg $s$ attached at $v_\infty(e)$ with monodromy type $m(s)=0$ a \textbf{broad marking}.(The leg $s$ is not necessarily of the type $(1,\rho)$ or $(1,\phi)$.) 

\begin{rema}\label{rema:string}
Since $v_0(e)$ is an unstable vertex, such that $e$ is the only edge connected to $v_0(e)$, at most one leg can be attached to $v_0(e)$. Although we add that leg to the graph $\bar{e}$, we can observe that the moduli $\cW_{\bar{e}}$ does not change.
\end{rema}

Note that $v_\infty(e)$ and $v_0(e)$ are unstable vertices of the graph $\bar{e}$. Let $\cW_{\bar{e}}$ be the moduli space of MSP fields $\xi$ such that $\Gamma_\xi = \bar{e}$. It is essential to show that the dimension of $\cW_{\bar{e}}$ is equal to its virtual dimension, and therefore has lci(locally complete intersection) singularity. For that, we compute the virtual dimension. We have:
\begin{align*}
vdim \cW_{\bar{e}} = & dim D_{\bar{e}} + \chi_T( \cL_1 \otimes \cN \otimes \bL_1) + \chi_T( \cN) + 3\chi_T(\cL_1(-\sigma^{\cC}_{(1,\phi)})) \\
& + 3 \chi_T(\cL_2) + \chi_T( \cL_1^{-3}\otimes \cL_2^{-3} \otimes \omega_{\cC}^{\log} (-\Sigma^{\cC}_{(1,\rho)})).
\end{align*}

\noindent We can check $\dim D_{\bar{e}} = -2 + -2 = -4$. Also, similar to the computation in Section \ref{sect:nostring}, we have
\[
\chi_T( \cL_1 \otimes \cN \otimes \bL_1) + \xi_T( \cN) = \chi( \cL_1 \otimes \cN \otimes \bL_1) + \chi( \cN) = 1 + 1 = 2
\]
Also, from a direct calculation we can check that 
$$
\chi_T( \cL_1(-\sigma^{\cC}_{(1,\phi)})) = \chi_T( \cL_2) = 1, \ \  \chi_T( \cL_1^{-3}\otimes \cL_2^{-3} \otimes \omega_{\cC}^{\log} (-\Sigma^{\cC}_{(1,\rho)})) = 0.
$$ 
Note that for the last equality, one should take into account that we have a broad marking at the vertex $v_\infty(e)$. Thus we have
\[ 
3\chi_T( \cL_1(-\sigma^{\cC}_{(1,\phi)})) + 3 \chi_T(\cL_2) + \chi_T(\cL_1^{-3}\otimes \cL_2^{-3} \otimes \omega_{\cC}^{\log} (-\Sigma^{\cC}_{(1,\rho)})) = 6.
\]
Hence we have
\[
\vdim \cW_{\bar{e}} = -4 + 2 + 6 = 4.
\]

Next we compute the dimension of $\cW_{\bar{e}}$. We will construct a bijective morphism from the coarse moduli space of $\cW_{\bar{e}}$, namely $|\cW_{\bar{e}}|$ to $\PP^2 \times \PP^2$. 
We first show that $\cC_e$ is a union of two rational curves $\cC_+$ and $\cC_-$ where $\cC_+$ meets $\cC_\infty$ and $\cC_-$ meets $\cC_0$. If $\cC_e$ is irreducible, then by Lemma \ref{lemm2}, we have $\cC_e \cong \PP^1$, $\cL_1 \cong \omega_{\cC_e}^{\mathrm{log}} \cong \cO_{\PP^1}$. But since there is only one marked points on $\cC_e$(corresponding to the unique leg), we have $\omega_{\cC_e}^{\mathrm{log}} \cong \cO_{\PP^1}(-1)$, which leads to a contradiction. 

Therefore the curve $\cC_e$ should be decomposed into $\cC_e = \cC_+ \cup \cC_-$. 
  Let $q_\infty := \cC_+ \cap \cC_0$ and $q_0 := \cC_- \cap \cC_0$, $x := \cC_+ \cap \cC_-$. Since the induced $\CC^*$--action on $\cC_+, \cC_-$ are nontrivial and since $\CC^*$--action on $\theta_1,\theta_2,\theta_3$ is trivial, the morphism $\cC_e \to \PP^2$ induced from $\theta$ should be a constant map. This gives us a map $g_1 : |\cW_{\bar{e}}| \to \PP^2$. On the other hand, $q_0 \in \cC_0$ and by the arguments in Section \ref{sect:MSPdescript} (a), $\phi$ is nonvanishing over the entire $\cC_0$. Hence $\phi|_{q_0} \in \cL_1^{\oplus 3}|_{q_0} \setminus \set{0}$. Therefore, the ratio $[\phi_1(q_0):\phi_2(q_0):\phi_3(q_0)]$ is an element of $\PP^2$. Therefore it defines another map $g_2 : |\cW_{\bar{e}}| \to \PP^2$. Therefore, we have a morphism
\[
(g_1,g_2) : |\cW_{\bar{e}}| \to \PP^2 \times \PP^2.
\]

\begin{lemm} \label{lemm:bij}
The morphism $(g_1,g_2)$ is bijective.
\end{lemm}
\begin{proof}
  First we show that it is surjective. Consider an element $([a_1 : a_2 : a_3], [b_1 : b_2 : b_3]) \in \PP^2 \times \PP^2$.
  Let $(\cL_2, \theta_1, \theta_2,\theta_3) = (\cO_{\cC_e}, a_1, a_2, a_3)$. Also let $\cL_1|_{\cC_-}  = \cO_{\PP^1}(d)$, $d>0$, using the isomorphism $\cC_- \cong \PP^1$. And choose the data:
\begin{align}\label{eq:const1}
( \phi_1|_{\cC_-}, \phi_2|_{\cC_-}, \phi_3|_{\cC_-}, \mu|_{\cC_-} ) = (b_1 x^{d}, b_2 x^{d}, b_3 x^{d}, y^{d})\
\end{align} 
  which induces the morphism $f_{\phi} : \PP^1 \cong \cC_- \to \PP^3$, a degree $d_\infty$ multiple cover of a line connecting $[1:0:0:0]$ and $[0:b_1:b_2:b_3]$ with exactly two ramification points $q_-=[1:0], x = [0:1]$ over $[1:0:0:0],[0:a_1:a_2:a_3]$. 
  Since the edge $e$ is obtained by flattening, we have $\cL_1|_{\cC_+} = \cO(-d)$ by Lemma \ref{lemm1}. Let $\cL_1$ over $\cC_e$ be a gluing of $\cL_1|_{\cC_+}$ and $\cL_1|_{\cC_-}$. Set $\phi_i|_{\cC_+} = 0$, then $\phi_i|_{\cC_+}$ and $\phi_i|_{\cC_-}$ glue to the sections $\phi_i$ of $\cL_1$. 

On the other hand, we let $\cN|_{\cC_-} = \cO_{\cC_-}$ and $\nu|_{\cC_-} = 1$. 
Let $\cN|_{\cC_+} = \cO_{\PP^1}(d_\infty)$ and let $\cN$ be a gluing of $\cN|_{\cC_-}$ and $\cN|_{\cC_+}$. We define $\nu|_{\cC_+} \in \cO_{\cC_+}(d_\infty)$ and $\rho|_{\cC_+} \in \cO_{\cC_+}(3 d_\infty)$
as $((\rho|_{\cC_+},\nu|_{\cC_+})) = (x^{3 d_\infty}, y^{d_\infty})$, which induces a branched covering to $\PP[1:3]$ with two branched points $[0:1],[1:0] \in \PP^1$ over $[0:1], [1:0] \in \PP[1:3]$. Then $\nu|_{\cC_-}, \nu|_{\cC_+}$ glue to the section $\nu$ of $\cN$.
Furthermore, we let $\rho|_{\cC_-} = 0$. Then $\rho|_{\cC_+},\rho|_{\cC_-}$ glue to the section $\rho \in \rH^0(\cL_1^{-3}\otimes \cL_2^{-3} \otimes \omega_{\cC}^{\log} )$. 

  From the above construction, we have $\cL_1 \otimes \cN|_{\cC_+} \cong \cO_{\cC_+}$ and let $\mu|_{\cC_+} = 1$. Then $\mu|_{\cC_-}$ defined in \eqref{eq:const1} and $\mu|_{\cC_+}$ glue to the element $\mu \in \rH^0(\cC_e, \cL_1 \otimes \cN)$.
  Therefore this choice $(\cC_e,\cL_1,\cL_2,\cN,\phi,\theta,\mu,\nu,\rho)$ defines an element $\xi \in \cW_{\bar{e}}$ such that $(g_1,g_2)(\xi_{a,b} )= ([a_1:a_2:a_3],[b_1:b_2:b_3])$. Hence $(g_1,g_2)$ is surjective.

\smallskip

To show injectivity, it is enough to show that $(g_1,g_2)^{-1}( \set{ ([a_1:a_2:a_3],[b_1:b_2:b_3]) } )$ is a point. Consider an element $\xi' = (\cC_e',\cL_1',\cL_2',\cN',\phi',\theta',\mu',\nu',\rho') \in (g_1,g_2)^{-1}([a_1:a_2:a_3],[b_1:b_2:b_3])$. 
We can easily observe $\cL_1' \cong \cL_1$, $\cL_2' \cong \cL_2$, $\cN' \cong \cN$, because $\cC_e = \cC_+ \cup \cC_-$ and the different gluing of line bundles at the node gives us the same line bundles up to isomorphisms. 

  From the description of ($0$--$1$)--edge and ($1$--$\infty$)--edge in Section \ref{sect:decompdomain}, $(\mu|_{\cC_-},\phi|_{\cC_-})$ defines a degree $d_\infty$ multiple cover of a line connecting $[1:0:0:0]$ and $[0:a_1:a_2:a_3]$, which has exactly two ramification points $q_-=[1:0], x = [0:1]$ over $[1:0:0:0],[0:a_1:a_2:a_3]$, and moreover, $(\rho|_{\cC_+}, \nu|_{\cC_+})$ induces a branched covering to $\PP[1:3]$ with two branch points $[0:1],[1:0] \in \PP^1$ over $[0:1], [1:0] \in \PP[1:3]$.
  From this, we have $\phi|_{\cC_\pm}$ and $\phi'|_{\cC_\pm}$ differ by scalar multiple. But this can be identified via the automorphism $\Aut(\cL_1|_{\cC_\pm})$. Therefore, $\phi,\phi'$ can be identified via $\Aut(\cL_1)$. Similarly, we can observe for other fields, $(\theta, \rho, \mu, \nu)$ and $(\theta', \rho', \mu', \nu')$ can be identified via automorphisms of line bundles, $\Aut(\cL_2), \Aut(\cN)$. Therefore, we have $\xi = \xi' \in \cW_{\bar{e}}$. Therefore $(g_1,g_2)$ is injective.
\end{proof}

Thus, $(g_1,g_2)$ is a bijective morphism and we have $\dim( \cW_{\bar{e}}) = 4$. Therefore, $\dim \cW_{\bar{e}} = \vdim \cW_{\bar{e}}$, and $\cW_{\bar{e}}$ has local complete intersection singularity since its perfect obstruction theory is (\'etale)locally comes from Kuranishi model \cite[p. 1037]{KL13}.
\begin{lemm}\label{lemm:lcising}
The moduli space of $\bar{e}$--framed MSP fields, $\cW_{\bar{e}}$ has locally complete intersection singularities.
\end{lemm}



  Let $\Gamma'$ be the graph obtained from $\Gamma$ by removing the edge $e$ and the vertex $v_0(e)$, and legs connected to $v_0(e)$ and attach a new leg $\ell$ decorated by $(1,\phi)$ to the vertex $v_\infty(e)$. 
By Remark \ref{rema:string}, regardless of $v_0(e)$ having legs attached to it or not, there is a morphism $\cW_{\Gamma} \to \cW_{\bar{e}}$ given by the restrictions of MSP fields $\xi = (\cC, \Sigma^{\cC}, \dots)$ over $\cC_e \subset \cC$. 
  Let $\cW_{\bar{e}}^\mu$ be the reduced substack of $\cW_{\bar{e}}$ and let $\cW_{\Gamma}^\mu := \cW_{\Gamma} \times_{\cW_{\bar{e}}} \cW_{\bar{e}}^\mu$. 
Let $\kappa : \cW^\mu_{\Gamma} \to \cW_{\Gamma'}$ be the morphism, which is obtained by restricting MSP fields to its subcurve. This morphism is well-defined for each MSP field $\xi \in \cW^\mu_{\Gamma}$ because its value of $\phi$ at the node $\cC_e \cap \cC_{v_\infty(e)}$ vanishes.

The reason for the vanishing is the following. By mimicking the argument in the paragraph preceding Lemma \ref{lemm:bij}, we can check $\cC_e$ decomposes into $\cC_e = \cC_+ \cup \cC_-$. By arguments in Section \ref{sect:MSPdescript} (d), $d := \deg \cL_1|_{\cC_-} > 0$ and by Lemma \ref{lemm1}, $\deg \cL_1|_{\cC_+} = -d < 0$. Therefore $\phi_i |_{\cC_+} \equiv 0$ so that $\phi_i$ vanishes over $\cC_e \cap \cC_{v_\infty}(e) = \cC_+ \cap \cC_{v_\infty}(e)$.

\begin{caut}
Note that in general, we cannot directly construct a morphism $\cW_{\Gamma} \to \cW_{\Gamma'}$ because universal MSP field over $\cW^\mu_{\Gamma}$ is not guaranteed to vanish over the divisor $R_{\ell}$ of the universal curve $\cC_{\Gamma}$ over $\cW_\Gamma$, where $R_{\ell}$ is the image of the section corresponding to the marking assigned to the new leg $\ell$.
\end{caut}

Recall that $\cW_{\Gamma'}^-$ denotes the degeneracy locus of the cosection $\sigma_{\Gamma} : Ob_{\cW_{\Gamma'}} \to \cO_{\cW_{\Gamma'}}$. 
Then we define $\cW_{\Gamma}^\sim := \cW_{\Gamma}^{\mu}  \times_{\cW_{\Gamma'} } \cW_{\Gamma'}^-$ and let $\wtil{\kappa} : \cW_{\Gamma}^\sim \lra \cW^-_{\Gamma'}$ be the induced morphism. Also, parallel to \cite[Section 5]{CL20van}, we can check $\kappa$ and $\wtil{\kappa}$ are flat morphisms.
Then, the following proposition is obtained by a parallel arguments of \cite[Proof of Proposition 5.6]{CL20van}.
\begin{prop}\label{prop:strforgetvirt} $\cW^\sim_\Gamma$ is proper and contains $\cW_{\Gamma}^-$ as a closed substack. Let $j : \cW_{\Gamma}^- \to \cW_{\Gamma}^\sim$ be the inclusion. Then, there exists a rational number $c \in \QQ$ such that
\begin{align}\label{eq:strforgetvirt}
j_*\left[ \cW_{\Gamma} \right]\virt\loc = c \cdot \wtil{\kappa}^* \left[ \cW_{\Gamma'} \right]\virt\loc \in A_*(\cW_{\Gamma}^\sim).
\end{align}
\end{prop}
Therefore, if $[\cW_{\Gamma'}]\virt\loc = 0$, we obtain $[\cW_\Gamma]\virt\loc \sim 0$. 
Hence, we can reduce our proof to the case that $\Gamma$ does not have strings. 
In summary, we obtain \eqref{eq:insepvan} as follows.
\begin{proof}[Proof of Equation \eqref{eq:insepvan}]
By the arguments in Section \ref{sect:decouptrim}, \ref{sect:forget}, \ref{sect:redtonost}, it is enough to show that $[\cW_\Gamma]\virt\loc = 0$ for the graphs $\Gamma$ satisfying the special condition in Section \ref{sect:nostring}. So we obtain the proof from \eqref{eq:specialvan}.
\end{proof}



\section{Proof for N-MSP fields}\label{sect:insepvanN}

N-MSP field is a generalization of MSP fields, which is introduced in \cite{CGLL}, and it is used to prove polynomiality of Gromov-Witten potentials and BCOV's Feynman rule of quintic 3-folds in \cite{CGL1, CGL2}.

Here we define N-MSP field by the following data: 
$$\xi = (\cC, \Sigma^{\cC}, \cL_1, \cL_2, \cN, \phi, \theta, \rho, \bm{\mu}=(\mu_1,\dots,\mu_N), \nu )$$ where
\begin{align*}
& \phi=(\phi_1,\phi_2,\phi_3) \in \rH^0( \cC, \cL_1 ^{\oplus 3} ), \ \theta=(\theta_1,\theta_2,\theta_3) \in \rH^0(\cC, \cL_2^{\oplus 3}), \ \rho \in \rH^0(\cC, \cL_1^{-3}\otimes \cL_2^{-3} \otimes \omega^{\log}_{\cC}  ), \\
& \mu_1,\dots,\mu_N \in \rH^0(\cC, \cL_1 \otimes \cN), \nu \in \rH^0( \cC, \cN)
\end{align*}

  The only thing different from the data of MSP field is that we use multiple $\mu$--fields $\bm{\mu}=(\mu_1,\dots,\mu_N)$ instead of the single $\mu$--field.
  Moreover we define the stability condition of N-MSP field $\xi$ to be the same as the stability condition in Definition \ref{cond:stab} by replacing $\mu$ with multiple $\mu$--fields, $\bm{\mu} = (\mu_1,\dots,\mu_N)$. 
  Parallel to the moduli space of MSP fields $\cW_{g,\bold{d},\gamma}$, we define a torus action on the moduli space of N-MSP field as follows. Consider a complex torus $T=(\CC^*)^N$. $T$ acts on the moduli space of N-MSP fields given by the following. For $\bold{t}=(t_1,\dots,t_N) \in T$, we define the action by
\begin{align*}
\bold{t} \cdot (\phi, \theta, \rho, \mu, \nu) & = (\phi, \theta, \rho, t_1\mu_1,\dots,t_N\mu_N, \nu).
\end{align*}

In \cite{CLLL}, the moduli space of stable N-MSP field is constructed as a DM stack in a general setting. Moreover it is proven that it has a natural $T$--equivariant perfect obstruction theory with $T$--invariant cosection $\sigma$ induced from the potential function $F(\bx,\by)\cdot p$, and moreover the degeneracy loci of the cosection $\sigma$ is proper.

As discussed in \cite{CGLL}, for a fixed N-MSP field $\xi \in \cW^T$, we have one more decoration over each vertex in the localization graph $v \in \Gamma_\xi$ called \textbf{`hour'}. That is, a number $\alpha_v \in \set{1,2,\dots, N}$ such that $\mu_i|_{\cC_v} = 0$ for $i \neq \alpha_v$.

Therefore, when we consider a localization graph $\Gamma$ satisfying the special condition, we require in Section \ref{sect:nostring}, the virtual dimension coming from the deformations of $\bm{\mu}$ is equal to the deformation of single $\mu$--field $\mu_{\alpha(v)}$ for each vertex $v \in \Gamma$. Therefore the virtual dimension coming from the deformation of $\bm{\mu},\nu$ is the same as \eqref{eq:vdimcount2}. Therefore, under the special condition we obtain the same conclusion that $\vdim \cW_\Gamma < 0$ and hence $[\cW_\Gamma]\virt\loc=0$.

The results in Section \ref{sect:decouptrim}, \ref{sect:forget}, \ref{sect:redtonost} can be directly generalized to the moduli space of N-MSP fields in an obvious way, as in \cite{CGLL}. As with \eqref{eq:insepvan}, when the graph $\Gamma$ is irregular and not a pure loop, we obtain:
\begin{align}\label{eq:insepvan2}
[\cW_\Gamma]\virt\loc \sim 0.
\end{align}

\section{Appendix: hybrid Landau-Ginzburg theory at infinity}

Consider the special case when a graph $\Gamma$ only consists of one vertex $v$ at infinity, with $k$ legs. Let $\cC_v$ be the corresponding curve and let $\{ x_1,\dots, x_k \}$ be marked points on $\cC_v$ corresponding to $k$ legs. 
Assume there are no $(1,\phi)$ markings on $\cC_v$. Let the monodromy type of $\cL_1$ be $\gamma_v = ( (\zeta_3)^{k-m}, (\zeta_3^2)^{m})$.

The moduli space $\cW_\Gamma$ is the moduli space parametrizing \textbf{mixed data}, 3-spin curve on $\cC_v$ and maps from $\cC_v$ to $\PP^2$ are combined. More precisely, $\cW_\Gamma$ parametrizes the following data:
\begin{enumerate}
\item
$\cC$ is a  genus $g$ orbifold curve with marked points $\Sigma^{\cC} = (x_1,\dots ,x_k)$, where $\Aut(x_i) \cong \mu_3$.
\item
Line bundle  $\cL_2$ over $\cC$ with degree $\deg \cL_2 = d'$,  and a section
$$\theta \in \rH^0(\cC, \cL_2^{\oplus 3})$$ nowhere vanishing, therefore inducing a map $f_{\theta} : \cC \to \PP^2$.
\item
Line bundle $\cL_1$   over $\cC$ with
an isomorphism:
\[
\rho : \cL_1^3 \otimes \cL_2^3 \stackrel{\cong}{\lra} \omega_{\cC,\log};
\]
\item 
$\Aut(x_i) \cong \mu_3$ act on $\cL_1|_{x_i}$ \textbf{faithfully} as:
\[
\zeta = \exp(2 \pi i / 3), \ \zeta \cdot \cL_1|_{x_i} = \zeta^{m_i} \cL_1|_{x_i},
\]
where each $m_i=1\,\text{or} \, 2$; we denote 
  $\gamma=(\zeta^{m_1},\dots, \zeta^{m_k})  $.
  \item
Sections $\phi \in \rH^0(\cC, \cL_1^{\oplus 3})$ 
 
\end{enumerate}
We call this data \textbf{LG}   fields.
We call $\cW_{g,1^{k-m} 2^{m},d'}^{LG}$ the moduli space of  such objects with monodromy type ( $(\zeta_3)^{k-m}, (\zeta_3^2)^{m})$.
Note that (4) implies $\cL_2\cong f_{\theta}\sta \cO(1)$ and $d'\in \ZZ_{\geq 0}$. Thus modulo $3$,
the isomorphism $\rho$ implies constraint
\begin{align}\label{m-type} m \equiv 2g-2 \ (\text{mod}\ 3)\end{align}

The virtual dimension is

$$\vdim \cW_{g,1^{k-m} 2^{m},d'}^{LG} = k.$$

We are particularly interested in the case $k-m=0$. In this case, for each    $m$ satisfying \eqref{m-type}
we define
\begin{align}\label{new-inv} N_{g,m,d'}:=\int_{[\cW_{g,2^{m},d'}^{LG} ]} 1 \in \QQ. \end{align}

 This is the curve counting for the partial-LG space $ ([\CC^{\oplus 3}/\mu_3]\ltimes \PP^2, W)$. For example
 if $g=0$, by  \eqref{m-type} we can have  $m=1,4,7,10,13,\cdots$, the genus zero potential is the series
 $$F_0:=  \sum_{m=1,4,7,\cdots }\sum_{d'=0}^\infty N_{0,m,d'}q_1^{m} q_2^{d'}  .$$

 If $g=1$, by  \eqref{m-type} we can have  $m=0,3,6,9,12,\cdots$, the genus one potential is the series
  $$F_1:=  \sum_{m=0,3,6,\cdots }\sum_{d'=0}^\infty N_{0,m,d'}q_1^{m} q_2^{d'} \ .$$
  
\smallskip

In Quintic $3$-fold case, the theory corresponding to a single vertex $v$ at infinity was Fan-Jarvis-Ruan-Witten(FJRW) theory \cite{FJR1, FJR2}. Recently it has been proved that potential function of FJRW invariants satisfies polynomiality and BCOV's Feynman rule, similar to the case of GW theory. So we think studying this hybrid LG theory will also be interesting.

\bibliographystyle{plain}

\end{document}